\documentclass[a4paper,10pt, reqno]{amsart}

\usepackage[
  margin=30mm,
  marginparwidth=25mm,     
  marginparsep=2mm,       
  bottom=25mm,
  ]{geometry}

\usepackage{amsfonts}
\usepackage{amssymb,amsthm,mathrsfs,amsmath,amscd,enumerate,enumitem, amscd,color}
\usepackage{xcolor}
\usepackage{tikz}   

\usepackage{lmodern}

\newtheorem{ThA}{Theorem}

\newtheorem{thm}{Theorem}[section]
 \newtheorem{cor}[thm]{Corollary}
 
 \newtheorem{prop}[thm]{Proposition}
\theoremstyle{definition}

 \theoremstyle{remark}
 \newtheorem{rem}[thm]{Remark}

\usepackage{hyperref}

\newcommand{\supp}{\mathop{\mathrm{supp}}}

\setlist[enumerate,1]{label=\textnormal{\bf (\alph*)}}

\numberwithin{equation}{section}
\allowdisplaybreaks

\begin{document}

\title[Higher order logarithms of Bessel operators and an extension problem]{Higher order logarithms of Bessel operators and an extension problem}
\author[J. J. Betancor]{J. J. Betancor}

\author[E. Dalmasso]{E. Dalmasso}

\author[J. C. Fari\~na]{J. C. Fari\~na}

\author[P. Quijano] {P. Quijano}

\address{Jorge J. Betancor, Juan Carlos Fari\~na\newline
	Departamento de An\'alisis Matem\'atico, Universidad de La Laguna,\newline
	Campus de Anchieta, Avda. Astrof\'isico S\'anchez, s/n,\newline
	38721 La Laguna (Sta. Cruz de Tenerife), Spain}
\email{jbetanco@ull.es, 
jcfarina@ull.edu.es}

\address{Estefan\'{\i}a Dalmasso, Pablo Quijano\newline
	Instituto de Matem\'atica Aplicada del Litoral, UNL, CONICET, FIQ\newline
    Colectora Ruta Nac. Nº 168, Paraje El Pozo,\newline
    S3007ABA, Santa Fe, Argentina}
\email{edalmasso@santafe-conicet.gov.ar, pquijano@santafe-conicet.gov.ar}


\renewcommand{\datename}{Last modified:}
\date{\today}


\begin{abstract}
We consider the Bessel operator defined by 
\[
B_\lambda =-\frac{d^2}{dx^2}+\frac{\lambda^2-1/4}{x^2},
\]
on $(0,\infty)$, with $\lambda>-1$. We study the fractional power $B_\lambda^s$, $s\in (-1,1)$, $s\neq 0$, and the logarithm $\log^kB_\lambda $, $k\in \mathbb N$, of $B_\lambda$. We obtain pointwise representations of these operators and asymptotic Taylor expansions of  the operators $B_\lambda^s$ in terms of logarithmic operators $\log^kB_\lambda $. We also obtain $\log B_\lambda$ as the solution of an extension problem.
\end{abstract}

\maketitle

\section{Introduction}
Consider the Bessel operator $B_\lambda$ defined on $(0,\infty)$ as
\[
B_\lambda =-\frac{d^2}{dx^2}+\frac{\lambda^2-1/4}{x^2},\quad \lambda >-1.
\]
It is clear that $B_\lambda=B_{-\lambda}$ when $\lambda\in (-1,1)$. The logarithmic operator of $B_\lambda$  has been studied recently by Fu, Li and Zhang \cite{FLZ}. Our objective is to define higher order logarithms of $B_\lambda$ and to derive the logarithmic operator $\log B_\lambda$ as a solution of an extension problem.

Given $\lambda >-1$, the Hankel transform $h_\lambda$, defined by
\[
h_\lambda (f)(x)=\int_0^\infty \sqrt{xy}J_\lambda (xy)f(y)dy,\quad x\in (0,\infty),
\]
plays the same role in the Bessel setting as the Fourier transform does for the Euclidean Laplacian. Here, $J_\lambda$ is the Bessel function of the first kind and order $\lambda$. 

If $f\in C_c^\infty (0,\infty)$, the space of smooth functions on $(0,\infty)$ with compact support, we have that $h_\lambda (B_\lambda f)(x)=x^2h_\lambda (f)(x)$, $x\in (0,\infty)$. Furthermore, $h_\lambda $ defines an isometric isomorphism in $L^2(0,\infty)$ (see \cite[Lemma 2.7]{BS} for a proof when $\lambda >-1$). 

We extend the definition of the Bessel operator $B_\lambda$ through the Hankel transform as follows
\begin{equation}\label{eq: Bessel extension}
{\mathcal B}_\lambda f=h_\lambda (y^2h_\lambda (f)),\quad f\in D(\mathcal B_\lambda),
\end{equation}
where $D(\mathcal B_\lambda)=\{f\in L^2(0,\infty): \,y^2h_\lambda (f)\in L^2(0,\infty)\}$.

If $f\in C_c^\infty (0,\infty)$ then $\mathcal B_\lambda f=B_\lambda f$. However $\mathcal B_\lambda\neq\mathcal{B}_{-\lambda}$ when $\lambda\in (-1,1)\setminus\{0\}$. In the sequel, we use $B_\lambda$ to denote also $\mathcal B_\lambda$ since we do not expect any confusion to arise. 

According to the definition \eqref{eq: Bessel extension} we can consider a functional calculus for the Bessel operator $B_\lambda$. If $\Phi$ is a complex measurable function on $(0,\infty)$ we define the operator $\Phi(B_\lambda)$ by
\begin{equation}\label{eq: spectral multiplier}
\Phi (B_\lambda)f=h_\lambda (\Phi (x^2)h_\lambda f),\quad f\in D(\Phi (B_\lambda))
\end{equation}
where 
\[
D(\Phi (B_\lambda))=\left\{f\in L^2(0,\infty): \Phi (x^2)h_\lambda f\in L^2(0,\infty)\right\}.
\]
The operators $\Phi (B_\lambda)$ are also called spectral multipliers for the Bessel operator $B_\lambda$. If $\Phi$ is a bounded function, $\Phi (B_\lambda)$ is bounded on $L^2(0,\infty)$. Several authors have established sufficient conditions on  $\Phi$ in order to obtain the boundedness of the operator $\Phi (B_\lambda)$ on $L^p(0,\infty)$, weighted $L^p(0,\infty)$ and Lorentz spaces (see, for instance, \cite{BMR, CS, GaS, GT, Gs, Gu, HNS, Ka} and \cite{ST}).

In this paper we study some operators that can be seen as spectral multipliers associated with the Bessel operators $B_\lambda$. We consider first fractional powers $B_\lambda^s$, $s\in \mathbb R\setminus\{0\}$, defined by
\[
B_\lambda^sf=h_\lambda (x^{2s}h_\lambda (f)),\quad f\in D(B_\lambda^s).
\]

We define, for every $t>0$, the operator $W_t^\lambda$ by
\[
W_t^\lambda (f)=h_\lambda (e^{-ty^2}h_\lambda (f)),\quad f\in L^2(0,\infty).
\]
The family $\{W_t^\lambda \}_{t>0}$ is a semigroup of operators in $L^2(0,\infty)$ whose infinitesimal generator is $-B_\lambda$. We can see that, for every $t>0$, $W_t^\lambda$  is the integral operator
\begin{equation}\label{eq: integral heat Bessel}
    W_t^\lambda (f)(x)=\int_0^\infty W_t^\lambda (x,y)f(y)dy,\quad f\in L^2(0,\infty),
\end{equation}
where 
\[
W_t^\lambda (x,y)=\frac{\sqrt{xy}}{2t}I_\lambda \left(\frac{xy}{2t}\right)e^{-\frac{x^2+y^2}{4t}},\quad x,y \in (0,\infty).
\]
Here, $I_\lambda$ represents the modified Bessel function of the first kind and order $\lambda$.

For every $t>0$ the integral in \eqref{eq: integral heat Bessel} defines a bounded operator in $L^p(0,\infty)$, $1\leq p\leq \infty$. Thus, $\{W_t^\lambda \}_{t>0}$ is a semigroup of operators in $L^p(0,\infty)$, $1\leq p\leq \infty$.

We now recall some properties of the modified Bessel functions $I_\lambda$, $\lambda >-1$, that will be useful in the sequel.

For every $\lambda >-1$, the modified Bessel function $I_\lambda$ of the first kind and order $\lambda$ is defined by
\begin{equation}\label{eq: def Ilambda}
I_\lambda (z)=\Big(\frac{z}{2}\Big)^\lambda \sum_{k=0}^\infty \frac{(z/2)^{2k}}{k!\Gamma (\lambda +k+1)},\quad z>0.
\end{equation}
From \eqref{eq: def Ilambda} we deduce that
\begin{equation}\label{eq: Ilambda at zero}
I_\lambda (z)\sim \frac{z^\lambda }{2^\lambda\Gamma (\lambda +1)},\quad \mbox{ as }z\rightarrow 0^+.
\end{equation}
According to \cite[(5.11.10)]{Le} we have that
\begin{equation}\label{eq: Ilambda at infty}
I_\lambda (z)=\frac{e^z}{\sqrt{2\pi z}}\left(\sum_{k=0}^n(-1)^k(\lambda, k)(2z)^{-k}+O(z^{-n-1})\right),\quad z>0.
\end{equation}
Here, $(\lambda ,0)=1$ and, for every $k\in \mathbb N$,
\[
(\lambda,k)=\frac{(4\lambda^2-1)(4\lambda^2-3^2)\cdots (4\lambda^2-(2k-1)^2)}{2^{2k}k!}.
\]
Taking derivatives in \eqref{eq: def Ilambda} leads to
\begin{equation}\label{eq: Ilambda derivative}
\frac{d}{dz}(z^{-\lambda }I_\lambda (z))=z^{-\lambda }I_{\lambda +1}(z),\quad z>0.
\end{equation}

Let $\lambda>-\frac{1}{2}$. According to the Feynman-Kac formula we have that
\begin{equation}\label{eq: heat Bessel FeynmanKac}
0\leq W_t^\lambda (x,y)\leq W_t(x-y)=\frac{e^{-\frac{|x-y|^2}{4t}}}{2\sqrt{\pi t}},\quad x,y,t\in (0,\infty),
\end{equation}
where $W_t(z)$, $z\in \mathbb R$ and $t>0$, denotes the Euclidean heat kernel in $\mathbb R$.

From \eqref{eq: Ilambda at zero} and \eqref{eq: Ilambda at infty} we deduce that
\begin{equation}\label{eq: heat Bessel t>xy}
0\leq W_t^\lambda(x,y) \leq Ce^{-\frac{x^2+y^2}{4t}}\frac{(xy)^{\lambda+1/2}}{t^{\lambda +1}},\quad x,y,t\in (0,\infty),\,t>xy,
\end{equation}
and
\begin{equation}\label{eq: heat Bessel comparison t<xy}
|W_t^\lambda (x,y)-W_t(x-y)|\leq C\sqrt{t}\frac{e^{-c\frac{|x-y|^2}{t}}}{xy}, \quad x,y,t\in (0,\infty),\,t<xy.
\end{equation}

In our first result we establish pointwise representations for the operators $B_\lambda^s$, $s\in (-1,1)$, $s\neq 0$. In the sequel, for every $x,r\in (0,\infty)$, we will denote $B(x,r)=(0,\infty)\cap (x-r,x+r)$.

\begin{thm}\label{thm: Th1.1}
Let $f\in C_c^\infty (0,\infty)$, $\lambda >-\tfrac12$ and $s\in (0,1)$. Then,
\begin{enumerate}
\item \label{itm: Th1.1 a} $f\in D(B_\lambda^s)$ and
\begin{align*}
    B_\lambda^sf(x)&=-\frac{s}{\Gamma (1-s)}\left[\int_0^\infty \left(f(y)-\mathcal X_{B(x,1)}(y)f(x)\right)\int_0^\infty \frac{W_t^\lambda (x,y)}{t^{1+s}}dtdy\right.\\
    &\quad +f(x)\Bigg(-\int_{(0,\infty)\setminus B(x,1)}\int_0^1\frac{W_t^\lambda (x,y)}{t^{1+s}}dtdy+\int_{B(x,1)}\int_1^\infty \frac{W_t^\lambda (x,y)}{t^{1+s}}dtdy\\
    &\quad \left.+\int_0^\infty \int_0^1\frac{W_t^\lambda (x,y)-W_t(x-y)}{t^{1+s}}dtdy\Bigg)\right]+\frac{f(x)}{\Gamma (1-s)},\quad \mbox{ a.e. }x\in (0,\infty). 
\end{align*}
\item \label{itm: Th1.1 b} when $0<s<\frac{\lambda}{2}+\frac{1}{2}$, $f\in D(B_\lambda^{-s})$ and
\begin{align*}
    B_\lambda^{-s}f(x)&=\frac{s}{\Gamma (1+s)}\left[\int_0^\infty \left(f(y)-\mathcal X_{B(x,1)}(y)f(x)\right)\int_0^\infty W_t^\lambda (x,y)t^{s-1}dtdy\right.\\
    &\quad +f(x)\Bigg(-\int_{(0,\infty)\setminus B(x,1)}\int_0^1W_t^\lambda (x,y)t^{s-1}dtdy+\int_{B(x,1)}\int_1^\infty W_t^\lambda (x,y)t^{s-1}dtdy\\
    &\quad +\left.\int_0^\infty \int_0^1(W_t^\lambda (x,y)-W_t(x-y))t^{s-1}dtdy\Bigg)\right]+\frac{f(x)}{\Gamma (1+s)},\quad \mbox{ a.e. }x\in (0,\infty).    
\end{align*}
\end{enumerate}
Furthermore, all the integrals are absolutely convergent for every $x\in (0,\infty)$.
\end{thm}
As we will see in the proof of Theorem~\ref{thm: Th1.1}, the representation formulas established there allow us to extend the definition of $B_\lambda^s$ for $s\in (-a,a)$, $s\neq 0$, where $0<a<1$ is sufficiently small, to a wider class of functions. 

Let $x\in (0,\infty)$ and $\theta\in (0,1]$. A function $f:(0,\infty)\rightarrow \mathbb{C}$ is said to be in ${\textup{Lip}}_{{\rm loc},x}(0,\infty)$ when
\[\|f\|_{{\textup{Lip}}^\theta_{{\rm loc},x}(0,\infty)}:=\sup_{y\in B(x,1)}\frac{|f(x)-f(y)|}{|x-y|^\theta}<\infty.
\]
We denote by ${\textup{Lip}}^\theta_{{\rm loc},{\rm uni}}(0,\infty)$ the space that consists of those functions $f$ defined on $(0,\infty)$ such that
\[
\|f\|_{{\textup{Lip}}^\theta_{{\rm loc},{\rm uni}}(0,\infty)}:=\sup_{x\in (0,\infty)}\|f\|_{{\textup{Lip}}^\theta_{{\rm loc},x}(0,\infty)}<\infty.
\]

The proof of Theorem~\ref{thm: Th1.1} (see Section~\ref{sec: proofs Th1.1 Th1.2}) also shows that the integrals appearing in the representation formulas are absolutely convergent when $f\in L^1\left((0,\infty),\frac{1}{1+y}dy\right)\cap {\textup{Lip}}^\theta _{{\rm loc},x}(0,\infty)$, with $\lambda\in (-\min\left\{\tfrac{\theta}{4},\frac{\lambda}{2}+\frac{1}{2}\right\},\tfrac{\theta}{4})$. Accordingly, if $x\in (0,\infty)$ and $f\in L^1\left((0,\infty),\frac{1}{1+y}dy\right)\cap {\textup{Lip}}^\theta _{{\rm loc}, x}(0,\infty)$ we define $B_\lambda^sf(x)$, $s\in (-\min\left\{\tfrac{\theta}{4},\frac{\lambda}{2}+\frac{1}{2 }\right\},\tfrac{\theta}{4})$, $s\neq 0$, by the representation formulas in Theorem~\ref{thm: Th1.1}.

In order to guarantee that $C_c^\infty(0,\infty)$ is contained in the domain of $B_\lambda^s$ for every $s\in\mathbb R$, and more generally in the domain of every logarithmic operator $\log^mB_\lambda$, $m\in\mathbb N$, that we introduce below, we recall a classical space of test functions due to Zemanian.

The function space $H_\lambda$ was defined by Zemanian in \cite{Ze} as follows. A smooth function $\phi$ on $(0,\infty)$ belongs to $H_\lambda$ if, for every $m,k \in \mathbb{N}\cup\{0\}$, 
\[
\gamma^\lambda_{m,k} (\phi) = \sup_{x\in (0,\infty)} \left| x^m\left(\frac{1}{x}\frac{d}{dx}\right)^k\left(x^{-\lambda - \tfrac12}\phi(x)\right)\right| < \infty.
\] 
It is clear that $C_c^\infty(0,\infty)$ is contained in $H_\lambda$.
The space $H_\lambda$ is a Fréchet space when the topology associated with the family of seminorms $\left\{\gamma^\lambda_{m,k}\right\}_{m,k\in \mathbb N}$ is considered on $H_\lambda$. Thus, the Hankel transform $h_\lambda$ is an isomorphism from $H_\lambda$ into itself (see \cite[Lemma 8]{Ze}) when $\lambda>-\tfrac12$.

We have that $H_\lambda \subset D(B_\lambda^s)$, $s \in \left(-\frac{\lambda}{2}-\frac{1}{2},+\infty\right)$. Indeed, let $\phi \in H_\lambda$. According to \cite[Lemma 8]{Ze}, $h_\lambda\phi\in H_\lambda$. Then,
\[
\int_0^\infty \left|x^{2s}h_\lambda(\phi)(x)\right|^2 dx \leq C\left(\int_0^1\left|x^{2s+\lambda+1/2}\right|^2dx + \int_1^\infty\left|\frac{x^{2s+\lambda+1/2}}{(1+x)^\ell}\right|^2dx\right) < \infty
\]
provided that $s> -\frac{\lambda}{2}-\frac{1}{2}$ and by taking $\ell \in \mathbb N_0$ and $\ell > 2s+\lambda+1$. Thus we proved that $\phi \in D(B_\lambda^s)$ when $s >-\frac{\lambda}{2}-\frac{1}{2}$.

In particular, $C_c^\infty(0,\infty)\subset H_\lambda\subset D(B_\lambda^s)$ for every $s\in\mathbb R$, a fact that we shall use throughout Section~\ref{sec: fract powers}.

Now we turn our attention to the logarithmic operators related to $B_\lambda$. By means of~\eqref{eq: spectral multiplier}, we define $\log^mB_\lambda$, $m\in \mathbb N$,  by
\[
\log^mB_\lambda =h_\lambda ((2\log x)^mh_\lambda (f)),\quad f\in D(\log^mB_\lambda).
\]
Recently, when $B_\lambda$ is replaced by the Laplacian $\Delta$ in $\mathbb R^d$, these operators  have been studied in \cite{Ch, ChHW} and \cite{ChW} (see also \cite{ChV2, ChV1, F, JSW, LW} and \cite{Le}). In the Laplacian case, the Fourier transform plays the analogous role to the Hankel transform.

In \cite[Theorem 1.1]{ChW} it was proved that if $f\in C_c^\infty (\mathbb R^d)$ then
\begin{equation}\label{eq: def log Delta}
(\log (-\Delta)f)(x)=\frac{d}{ds}(-\Delta )^s\Big|_{s=0}f(x)=c_d\int_{\mathbb R^d}\frac{f(x)\mathcal X_{B(x,1)}(y)-f(y)}{|x-y|^d}dy+\rho _df(x),\quad x\in \mathbb R^d,   
\end{equation}
where $c_d=\Gamma (d/2)\pi ^{-d/2}$ and $\rho _d=2\log2+ \psi(d/2)-\gamma$, being $\gamma=-\Gamma'(1)$ the Euler-Macheroni's constant and $\psi =\Gamma '/\Gamma $ the Digamma function. The representation \eqref{eq: def log Delta} allows us to extend the definition of $\log (-\Delta)$ to a larger class of functions satisfying Lipschitz and integrability conditions.

As we mentioned above, Fu, Li and Zhang \cite{FLZ} have studied the logarithm of the Bessel operator $B_\lambda$. In their notation, our operator $B_\lambda$ is denoted by $\Delta _{\lambda+1/2}$. Since $C_c^\infty(0,\infty)\subset H_\lambda\subset D(\log B_\lambda)$, their definition of $\log B_\lambda$ applies in particular to functions in $H_\lambda$. 

The main result in \cite{FLZ} is the following one.
\begin{ThA}[{\cite[Theorem 2.4]{FLZ}}]\label{thm: ThA}  Let $f\in C_c^\infty (0,\infty)$ and $\lambda \geq 0$. Then,
\begin{align}\label{eq: def log Bessel FLZ}
    (\log B_\lambda)f(x)&=\frac{d}{ds}B_\lambda^s\Big|_{s=0^+}f(x)\nonumber\\
    &=\int_0^\infty (f(x)\mathcal X_{B(x,1)}(y)-f(y))\int_0^\infty \frac{W_t^\lambda(x,y)}{t} dtdy+\rho (x)f(x),\quad x\in (0,\infty),
\end{align}
where
\begin{align*}
    \rho (x)&=\int_0^1\int_{(0,\infty)\setminus B(x,1)}\frac{W_t^\lambda (x,y)}{t}dydt-\int_1^\infty \int_{B(x,1)}\frac{W_t^\lambda (x,y)}{t}dydt-\int_0^1\frac{W_t^\lambda(1)(x)-1}{t}dt-\gamma.
\end{align*}

Furthermore, if $1<p<\infty$, $(\log B_\lambda)f\in L^p(0,\infty)$ and
\[
\lim_{s\rightarrow 0^+}\frac{B_\lambda^sf-f}{s}=(\log B_\lambda)f,\quad \mbox{ in }L^p(0,\infty).
\]    
\end{ThA}
The representation formula \eqref{eq: def log Bessel FLZ} enables us to define $(\log B_\lambda)f$ for functions $f$ in ${\textup{Lip}}^\theta(0,\infty)$, with $\theta\in(0,1]$, having compact support in $(0,\infty)$.

Unlike the representation in \eqref{eq: def log Delta} for $\log(-\Delta)$, where the corrector factor is constant, the corresponding factor in \eqref{eq: def log Bessel FLZ} depends on $x\in(0,\infty)$. The reason is that the Euclidean heat semigroup $\{W_t\}_{t>0}$ is Markovian, namely, $W_t(1)=1$, $t>0$,  whereas the Bessel heat semigroup $\{W_t^\lambda\}_{t>0}$ is not. This lack of Markovian property introduces additional difficulties in proving \eqref{eq: def log Bessel FLZ} and leads to a more intricate corrector factor.

Our purpose is to continue the study about the logarithm of $B_\lambda$ in two directions. We first define higher order logarithm operator $\log^mB_\lambda$, $m\in \mathbb N$, and we establish asymptotic expansions of the fractional powers $B_\lambda^s$ and $B_\lambda^{-s}$, as $s\rightarrow 0^+$, through higher order logarithm operators of $B_\lambda$. Secondly, we establish a representation of $\log B_\lambda $ by solving an extension problem.

In the following we prove a representation for $\log^m B_\lambda $, $m\in \mathbb N$.

\begin{thm}\label{thm: Th1.2}
Let $\lambda >-\tfrac12$, $f\in C_c^\infty(0,\infty)$ and $m\in \mathbb N$. Then, for almost every $x\in (0,\infty)$, 
\begin{align*}
    (\log^m B_\lambda)f(x)&=-m\sum_{j=0}^{m-1}(-1)^j\binom{m-1}{j}\partial_s^{m-j-1}\left(\frac{1}{\Gamma (1-s)}\right)\Bigg|_{s=0}\\
    &\hspace{-0.5cm}\times \left[\int_0^\infty (f(y)-\mathcal X_{B(x,1)}(y)f(x))\int_0^\infty W_t^\lambda (x,y)\frac{(\log t)^j}{t}dtdy\right.\\
    &\hspace{-0.5cm}+f(x)\left(-\int_{(0,\infty)\setminus B(x,1)}\int_0^1W_t^\lambda (x,y)\frac{(\log t)^j}{t}dtdy+\int_{B(x,1)}\int_1^\infty W_t^\lambda (x,y)\frac{(\log t)^j}{t}dtdy\right.\\
&\hspace{-0.5cm}+\left.\left.\int_0^\infty \int_0^1(W_t^\lambda (x,y)-W_t(x-y))\frac{(\log t)^j}{t}dtdy\right)\right]+f(x)\partial_s^{m}\left(\frac{1}{\Gamma (1-s)}\right)\Bigg|_{s=0}.
\end{align*}
Furthermore, all the above integrals are absolutely convergent for every $x\in (0,\infty)$.

We also have that
\[
(\log^m B_\lambda)f(x)=\partial_s^m B_\lambda^s(f)(x)=(-1)^m\partial_s^m B_\lambda^{-s}(f)(x)\Big|_{s=0},\quad a.e.\quad x\in (0,\infty).
\]
\end{thm}
This result represents a generalization of Theorem~\ref{thm: ThA}. It is worth noting that our proof of Theorem~\ref{thm: Th1.2} is shorter than the one of Theorem~\ref{thm: ThA} in \cite{FLZ}.

Arguing as in the proof of Theorem~\ref{thm: Th1.2}, one sees that the integrals in its representation formulas are absolutely convergent whenever $f\in L^1\!\left((0,\infty),\frac{dy}{1+y}\right)\cap
{\textup{Lip}}^\theta_{{\rm loc},x}(0,\infty),$
for some $\theta\in(0,1]$ (see Remark~\ref{rem: log}). Accordingly, if $x\in(0,\infty)$, $m\in\mathbb N$, and
$f\in L^1\!\left((0,\infty),\frac{dy}{1+y}\right)\cap
{\textup{Lip}}^\theta_{{\rm loc},x}(0,\infty),$
we define $(\log^m B_\lambda)f(x)$ by the representation formula in Theorem~\ref{thm: Th1.2}.
We will prove in Propositions~\ref{prop: Prop2.3} and~\ref{prop: Prop2.4} that, if $s_0\in(0,1)$ and $m_0\in\mathbb N$, then $\left.\partial_s^m B_\lambda^sf\right|_{s=0}
=(\log^mB_\lambda)f,$
for every $f\in D(B_\lambda^{s_0})\cap D(\log^{m_0}B_\lambda)$, $0<s<s_0$, and $m\in\mathbb N$ with $m\le m_0$, and
$\left.\partial_s^m B_\lambda^{-s}f\right|_{s=0}
=(-1)^m(\log^m B_\lambda)f,$
for every $f\in D(B_\lambda^{-s_0})\cap D(\log^{m_0}B_\lambda)$, $0<s<s_0$, and $m\in\mathbb N$ with $m\le m_0$. In both cases, $\partial_s$ is understood in the $L^2(0,\infty)$ sense.

Chen (\cite{Ch}) studied the higher-order logarithmic operators $\log^m(-\Delta)$, $m\in \mathbb N$. Motivated by the equality
\[
r^\sigma =1+\sum_{n=1}^m\frac{\sigma^n}{n!}(\log r)^n+o(\sigma ^m),\quad \mbox{ as }\sigma \rightarrow 0^+,
\]
Chen established in \cite[Theorems~1.1 and~1.2]{Ch} asymptotic expansions of the fractional operators $(-\Delta)^s$ and $(-\Delta )^{-s}$, as $s\rightarrow 0^+$, in terms of the logarithmic operators $\log^m(-\Delta)$, $m\in \mathbb N$.

Our next result establishes the corresponding Taylor asymptotic expansions for $B_\lambda^s$ and $B_\lambda^{-s}$ in terms of the logarithmic operators $\log^m(B_\lambda)$.

\begin{thm}\label{thm: Th1.3}
Let $\theta \in (0,1]$, $f\in C_c(0,\infty)\cap {\textup{Lip}}^\theta _{{\rm loc},{\rm uni}}(0,\infty)$, $1<p\leq \infty $, $\lambda >-\tfrac12$ and $m\in \mathbb N$. Then,
\[\lim_{s\rightarrow 0^+}\frac{(m+1)!}{s^{m+1}}\left(B_\lambda^sf-\left(f+\sum_{j=1}^m\frac{s^j}{j!}(\log^jB_\lambda)f \right)\right)=(\log^{m+1}B_\lambda)f,\quad \mbox{ in }L^p(0,\infty);\]
\[\lim_{s\rightarrow 0^+}\frac{(m+1)!}{s^{m+1}}\left(B_\lambda^{-s}f-\left(f+\sum_{j=1}^m\frac{(-1)^js^j}{j!}(\log^jB_\lambda)f\right)\right)=(-1)^{m+1}(\log^{m+1}B_\lambda)f, \mbox{ in }L^p(0,\infty).\]
\end{thm}

Caffarelli and Silvestre (\cite{CaSi}) established a celebrated result known as Extension Theorem where the authors gave a method that allows them to look at a nonlocal operator $(-\Delta )^s$, $0<s<1$, when it acts on a function $f$, through the solution of an initial value problem associated with an extension of the function $f$ to the half space $\mathbb R_+^{d+1}$. This procedure transfers a nonlocal problem to a local problem of higher dimension. Stinga and Torrea (\cite{ST}) solved the extension problem for fractional operators in Hilbert spaces by using the semigroup language. Later, Gal\'e, Miana and Stinga (\cite{GMS}) characterized through an extension problem the fractional powers of generators of integrated semigroups. Extension theorems have been established in other contexts (\cite{AEW, AH, BGS, BP} and \cite{KM}).

Chen, Hauer and Weth proved in \cite[Theorem 1.2]{ChHW} an extension theorem for the logarithmic operator $\log(-\Delta)$. An earlier related result can be found in \cite{GoSp}.

Now we shall define the logarithmic operator $\log B_\lambda$ through an extension problem as well. In order to do so, we start with some definitions. We say that $f\in L_{\rm loc}^1(0,\infty)$ has at most algebraic growth in the integral sense, in short, $f\in AG(0,\infty)$, when there exist constants $C,\sigma >0$ such that
\[
\|f\|_{L^1(B(x,1))}\leq C(1+x)^\sigma ,\quad x\in (0,\infty).
\]
In a similar way we define the function space $AG((0,\infty)\times (0,\infty))$.

A function $f:(0,\infty)\to \mathbb C$ is said to be Dini continuous in a point $x\in (0,\infty)$ when 
\[
\int_0^1\frac{w_{f,x}(r)}{r}dr<\infty,
\]
where $w_{f,x}(r)=\sup_{y \in B(x,r)}|f(y)-f(x)|$, $r\in (0,1)$. Moreover, a function $f:(0,\infty)\to \mathbb R$ is called uniformly Dini continuous in $(0,\infty)$ when 
\[
\int_0^1\frac{w_f(r)}{r}dr<\infty,
\]
where $w_f(r)=\sup_{x\in (0,\infty)}w_{f,x}(r)$, $r\in (0,1)$.

If $f\in L^1\big((0,\infty),\frac{dy}{1+y}\big)$ we define the extension $u_f^\lambda$ of $f$ to $(0,\infty)\times (0,\infty)$ by
\[
u_f^\lambda(x,t)=\int_0^\infty W_u^\lambda (f)(x)\frac{e^{-t^2/(4u)}}{u}du,\quad x,t\in (0,\infty).
\]

We are now in position to establish the following extension theorem.

\begin{thm}\label{thm: Th1.4}
Suppose that $f\in L^1\big((0,\infty),\frac{dy}{1+y}\big)\cap {\textup{Lip}}^\theta_{{\rm loc}, {\rm uni}}(0,\infty)$, $\theta\in (0,1]$, $\sigma>0$, and $\lambda> -\tfrac12$. Then 
\begin{enumerate}
    \item \label{itm: Th1.4 a} $u_f^\lambda \in C^2((0,\infty)\times (0,\infty))\cap AG((0,\infty)\times (0,\infty))$, $u_f^\lambda (\cdot , t)\in L^1\big((0,\infty),\frac{dy}{1+y^{\sigma +1}}\big)$ and 
    \[
    \|u_f^\lambda (\cdot ,t)\|_{L^1\big((0,\infty),\frac{dy}{1+y^{\sigma +1}}\big)}\leq C(1+\log^-t),\quad t>0,
    \]
where $\log^-t=\max\{0,-\log t\}$, $t>0$.

\item \label{itm: Th1.4 b} We have that $(\partial_t^2+\frac{1}{t}\partial_t)u_f^\lambda (x,t)=B_{\lambda ,x}u_f^\lambda(x,t)$, $x,t\in (0,\infty)$, where $B_{\lambda,x}$ denotes the operator $B_\lambda$ acting on the variable $x$. Moreover,
\[
\lim_{t\rightarrow +\infty}u_f^\lambda (x,t)=0,\quad x\in (0,\infty).
\]

\item \label{itm: Th1.4 c} The following limits hold:
\[
\lim_{t\to 0^+}t\partial_tu_f^\lambda(\cdot,t)=\lim_{t\rightarrow 0^+}\frac{u_f^\lambda (\cdot ,t)}{\log t}=-2 f,\quad \mbox{ in }L^1_{\rm loc}(0,\infty),
\]
that is, for every $K$ compact subset in $(0,\infty)$, 
\[
\lim_{t\rightarrow 0^+}\|t\partial_tu_f^\lambda(\cdot,t)+f(\cdot)\|_{L^1(K)}=\lim_{t\rightarrow 0^+}\left\|\frac{u_f^\lambda (\cdot ,t)}{\log t}+2 f(\cdot)\right\|_{L^1(K)}=0.
\]

\item \label{itm: Th1.4 d} In the distributional sense,
\[
( \log B_\lambda)f=-\lim_{t\rightarrow 0^+}(u_f^\lambda (x,t)+f(x)\log t)+f(x)\mathcal K(x),
\]
where $\mathcal K$ is given by
\begin{align*}
    \mathcal K(x)&=\int_{(0,\infty)\setminus B(x,1)}\int_0^1\frac{W_u^\lambda (x,y)-W_u(x-y)}{u}dudy\\
    &\quad +\int_{B(x,1)}\int_1^\infty \frac{W_u^\lambda (x,y)-W_u(x-y)}{u}dudy\\
    &\quad +\gamma-2\log x+\mathcal X_{(0,1)}(x)\log x+q+\widetilde{q}, \quad x\in (0,\infty),
\end{align*}
being $q=2\int_1^\infty ((1+r^2)^{-1/2}-r^{-1})dr=2\log\big(2/(1+\sqrt{2})\big)$ and $\widetilde{q}=2\int_0^1((1+t)t)^{-1/2}dt=4\log(1+\sqrt{2})$.

Furthermore, if $\phi$ is uniformly Dini continuous with compact support in $(0,\infty)$, then
\[
\int_0^\infty f(x)(\log B_\lambda)\phi (x)dx=-\lim_{t\rightarrow 0^+}\int_0^\infty \phi(x)(u_f^\lambda (x,t)+f(x) \log t)dx+\int_0^\infty \mathcal K(x)f(x)\phi (x)dx.
\]

\item \label{itm: Th1.4 e} If $f$ is a Dini continuous function at a point $x\in (0,\infty)$, then
\[
(\log B_\lambda)(f)(x)=-\lim_{t\rightarrow 0^+}(u_f^\lambda (x,t)+f(x)\log t)+\mathcal K(x)f(x).
\]
\end{enumerate}
\end{thm}

This paper is organized as follows. In Section~\ref{sec: fract powers} we obtain the representations of the operator $B_\lambda^s$, $s\in (-1,1)$, $s\neq 0$, by using the semigroup of operators $\{W_t^\lambda \}_{t>0}$. Theorems \ref{thm: Th1.1} and \ref{thm: Th1.2} are proved in Section~\ref{sec: proofs Th1.1 Th1.2}. Finally, a proof of Theorem~\ref{thm: Th1.3} is given in Section~\ref{sec: proof Th1.3} and in Section~\ref{sec: proof Th1.4} we prove Theorem~\ref{thm: Th1.4}.

Throughout this paper we always denote by $c$ and $C$ positive constants that can change in each occurrence.


\section{Fractional powers of Bessel operators}\label{sec: fract powers}

In this section we establish representations through the semigroup $\{W_t^\lambda \}_{t>0}$ for the fractional powers $B_\lambda^s$, $s\in (-1,1)$, $s\neq 0$, of $B_\lambda$.

Let $s>0$. We define the operators $\mathbb B_{\lambda ,s}$ and $\mathbb B_{\lambda ,-s}$ as follows:
\[
\mathbb B_{\lambda, s}f=\frac{1}{\Gamma (-s)}\int_0^\infty \frac{W_t^\lambda (f)-f}{t^{1+s}}dt,\quad f\in D(\mathbb B_{\lambda, s}),
\]
where
\[
D(\mathbb B_{\lambda, s})=\left\{f\in L^2(0,\infty): \int_0^\infty \frac{\|W_t^\lambda (f)-f\|_{L^2(0,\infty)}}{t^{1+s}}dt<\infty\right\},
\]
and 
\[
\mathbb B_{\lambda, -s}f=\frac{1}{\Gamma (s)}\int_0^\infty W_t^\lambda (f)t^{s-1}dt,\quad f\in D(\mathbb B_{\lambda, -s}),
\]
being
\[
D(\mathbb B_{\lambda, -s})=\left\{f\in L^2(0,\infty): \int_0^\infty \|W_t^\lambda (f)\|_{L^2(0,\infty)}t^{s-1}dt<\infty\right\}.
\]
Note that the function $F:[0,\infty)\to L^2(0,\infty)$ given by $F(u)=W_u^\lambda (f)$, with $f\in L^2(0,\infty)$, is continuous, because $\{W_t^\lambda \}_{t>0}$ is a $C_0$-semigroup in $L^2(0,\infty)$. Since $L^2(0,\infty)$ is a separable Banach space, according to Pettis' Theorem (\cite[p. 131]{Yo}), the function $F$ is strongly measurable. Hence, if $f\in D(\mathbb B_{\lambda ,s})$ (respectively, $D(\mathbb B_{\lambda ,-s})$) the function $t\in (0,\infty)\mapsto (W_t^\lambda (f)-f)/t\in L^2(0,\infty)$ (respectively, $t\in (0,\infty)\mapsto W_t^\lambda (f)t^{s-1}\in L^2(0,\infty)$) is $L^2(0,\infty)$-Bochner integrable in $(0,\infty)$. According to \cite[(3), p. 260]{Yo}, $B_\lambda^sf=\mathbb B_{\lambda ,s}f$, $f\in D(B_\lambda)$. 

We now establish relations between $B_\lambda^s$ and $B_\lambda^{-s}$ with $\mathbb B_{\lambda , s}$ and $\mathbb B_{\lambda , -s}$, respectively.

\begin{prop}\label{prop: Prop2.1}
  Let $\lambda >-1$ and $s_0\in (0,1)$. If $f\in D(B_\lambda^{s_0})$, then $f\in D(\mathbb B_{\lambda ,s})$ and $\mathbb B_{\lambda ,s}f=B_\lambda^sf$, $s\in (0,s_0)$.  
\end{prop}
\begin{proof}
Let $s\in (0,s_0)$. Suppose that $f\in D(B_\lambda^{s_0})$. We have that $f\in L^2(0,\infty)$ and $y^{2{s_0}}h_\lambda (f)\in L^2(0,\infty)$. Then, $f\in D(B_\lambda^s)$. Since $h_\lambda $ is an isometry in $L^2(0,\infty)$ we get
\[
\|W_t^\lambda (f)-f\|_{L^2 (0,\infty)}=\|h_\lambda (W_t^\lambda (f))-h_\lambda (f)\|_{L^2(0,\infty)}=\|(e^{-ty^2}-1)h_\lambda (f)\|_{L^2(0,\infty)},\quad t>0.
\]
Then,
\begin{align*}
    \int_0^\infty \frac{\|W_t^\lambda (f)-f\|_{L^2(0,\infty)}}{t^{1+s}}dt&\leq C\left(\int_0^1\frac{\|(ty^2)^{s_0}h_\lambda (f)\|_{L^2(0,\infty)}}{t^{1+s}}dt+\int_1^\infty \frac{\|h_\lambda (f)\|_{L^2(0,\infty)}}{t^{1+s}}dt\right)\\
    &\leq C\left(\|y^{2s_0}h_\lambda (f)\|_{L^2(0,\infty)}+\|f\|_{L^2(0,\infty)}\right).
\end{align*}
It follows that $f\in D(\mathbb B_{\lambda ,s})$, and using properties of the Bochner integrals we get
\begin{align*}
\mathbb B_{\lambda ,s}f&=\frac{1}{\Gamma (-s)}\int_0^\infty \frac{W_t^\lambda (f)-f}{t^{1+s}}dt=\frac{1}{\Gamma (-s)}\int_0^\infty\frac{h_\lambda (e^{-y^2t}h_\lambda (f))-h_\lambda (h_\lambda (f))}{t^{1+s}}dt\\
    &=h_\lambda \left(\frac{1}{\Gamma (-s)}\int_0^\infty \frac{e^{-ty^2}-1}{t^{1+s}}dt\,h_\lambda (f)\right)=h_\lambda (y^{2s}h_\lambda (f))=B_\lambda^sf,
\end{align*}
because $\int_0^\infty \frac{e^{-zt}-1}{t^{1+s}}dt=z^s\Gamma (-s)$, $z>0$.
\end{proof}

\begin{prop}\label{prop: Prop2.2}
 Let $\lambda >-1$ and $s_0\in (0,1)$. If $f\in D(B_\lambda^{-s_0})$, then $f\in D(\mathbb B_{\lambda ,-s})$ and $\mathbb B_{\lambda ,-s}f=B_\lambda^{-s}f$, $s\in (0,s_0)$.     
\end{prop}
\begin{proof}
    Let $s\in (0,s_0)$. Assume that $f\in D(B_\lambda^{-s_0})$. Then, $f\in D(B_\lambda^{-s})$. We get
\[
\|W_t^\lambda (f)\|_{L^2 (0,\infty)}=\|h_\lambda (e^{-ty^2}h_\lambda (f))\|_{L^2(0,\infty)}=\|e^{-ty^2}h_\lambda (f)\|_{L^2(0,\infty)},\quad t>0.
\]
It follows that
\begin{align*}
    \int_0^\infty \frac{\|W_t^\lambda (f)\|_{L^2(0,\infty)}}{t^{1-s}}dt&=\int_0^\infty \frac{\|e^{-ty^2}h_\lambda (f)\|_{L^2(0,\infty)}}{t^{1-s}}dt\\
    &\leq \|f\|_{L^2(0,\infty)}\int_0^1\frac{dt}{t^{1-s}}+C\int_1^\infty \frac{\|y^{-2s_0}h_\lambda (f)\|_{L^2(0,\infty)}}{t^{1-s+s_0}}dt\\
    &\leq C\Big(\|f\|_{L^2(0,\infty)}+\|y^{-2s_0}h_\lambda (f)\|_{L^2(0,\infty)}\Big).
\end{align*}
Then, $f\in D(\mathbb B_{\lambda ,-s})$ and 
\[
\mathbb B_{\lambda ,-s}f=\frac{1}{\Gamma (s)}\int_0^\infty \frac{W_t^\lambda (f)}{t^{1-s}}dt=h_\lambda \left(\frac{1}{\Gamma (s)}\int_0^\infty e^{-ty^2}t^{s-1}dt\,h_\lambda (f)\right)=h_\lambda (y^{-2s}h_\lambda (f))=B_\lambda^{-s}f.\qedhere
\]
\end{proof}
We now study the derivatives $\partial_s^m B_\lambda^s$ of $B_\lambda^s$ in $L^2(0,\infty)$ with $m\in \mathbb N$ and $s\in (-1,1)$, $s\neq 0$.

\begin{prop}\label{prop: Prop2.3}
    Let $\lambda >-1$, $s_0\in (0,1)$ and $m_0\in \mathbb N$. Suppose that $f\in D(B_\lambda^{s_0})\cap D(\log^{m_0}B_\lambda)$. Then
    \[
    \partial_s^m B_\lambda^sf=(\log^m B_\lambda)B_\lambda^sf,\quad s\in (0,s_0)\mbox{ and }m\in \mathbb N,\,m\leq m_0.
    \] 
 Here, $\partial_s$ is understood in $L^2(0,\infty)$.
\end{prop}

\begin{proof}
    Let $s\in (0,s_0)$ and let $m\in \mathbb N$, $m \leq m_0$. Since $D( B_\lambda^{s_0})\subseteq D(B_\lambda^s)$, $f\in D(B_\lambda^s)$ and $B_\lambda^sf=h_\lambda (y^{2s}h_\lambda (f))$. Moreover, we have that
    \begin{align*}
        |\log y|^m y^{2s}&\leq C\left(|\log y|^m\mathcal X_{(0,\tfrac12)}(y)+y ^{2s_0}\mathcal X_{(\tfrac12,\infty)}(y)\right)\\
        &\leq C\left(|\log y|^{m_0}+y^{2s_0}\right),\quad y\in (0,\infty).
    \end{align*}
Then, $B_\lambda^sf\in D(\log^m B_\lambda)$.

By using the mean value theorem we get
\[
|y^{2h}-1|\leq C\left(1+y^{2(s_0-s)}\right)h\log y,\quad y\in (0,\infty)\mbox{ and }h\in (0,s_0-s).
\]
Since $B_\lambda^sf\in D(\log B_\lambda)$ and $B_\lambda^{s_0-s}f\in D(\log B_\lambda)$, by using the dominated convergence theorem we obtain
\[
\lim_{h\rightarrow 0^+}\int_0^\infty \left|\left(\frac{y^{2h}-1}{h}-2\log y\right)y^{2s}h_\lambda (f)(y)\right|^2dy=0.
\]
We also have that
\[
|y^{-2h}-1|\leq 2h(1+y^{-s})\log y,\quad y\in (0,\infty)\mbox{ and }h\in \big(0,\tfrac{s}{2}\big),
\]
and
\[
\lim_{h\rightarrow 0^-}\int_0^\infty \left|\left(\frac{y^{2h}-1}{h}-2\log y\right)y^{2s}h_\lambda (f)(y)\right|^2dy=0.
\]
Thus, we proved that
\[
\lim_{h\rightarrow 0}\frac{B_\lambda^{s+h}f-B_\lambda^sf}{h}=(\log B_\lambda)B_\lambda^sf,\quad \mbox{ in }L^2(0,\infty).
\]
By proceeding in a similar way we can see that if $m\in \mathbb N$, $m\leq m_0-1$, and $\partial_s^m B_\lambda^sf=(\log^m B_\lambda)B_\lambda^sf$, $0<s<s_0$, then $\partial_s^{m+1}B_\lambda^sf=(\log^{m+1}B_\lambda)B_\lambda^s f$, $0<s<s_0$.

Thus the proof is finished.
\end{proof}
By arguing as in the proof of the previous proposition we obtain the following result.

\begin{prop}\label{prop: Prop2.4}
    Let $\lambda >-1$, $s_0\in (0,1)$ and $m_0\in \mathbb N$. Suppose that $f\in D(B_\lambda^{-s_0})\cap D(\log^{m_0}B_\lambda)$. Then, 
    \[
    \partial_s^m B_\lambda^{-s}f=(-1)^m(\log^m B_\lambda)B_\lambda^{-s}f,\quad s\in (0,s_0)\mbox{ and }m\in \mathbb N,\,m\leq m_0,
    \] 
    where $\partial_s$ is understood in $L^2(0,\infty)$.
\end{prop}

From Propositions \ref{prop: Prop2.3} and \ref{prop: Prop2.4} we can deduce the following property that can be established using the dominated convergence theorem.

\begin{cor}\label{cor: Cor2.5}
    Let $\lambda >-1$, $s_0\in (0,1)$, and $m_0\in \mathbb N$.
\begin{enumerate} 
    \item If $f\in D(B_\lambda^{s_0})\cap D(\log^{m_0}B_\lambda)$ then for every $m\in \mathbb N$, $m\leq m_0$,
\[
(\log^m B_\lambda)f=\lim_{s\rightarrow 0^+}\partial_s^m B_\lambda^sf,\quad \mbox{ in }L^2(0,\infty).
\]
\item If $f\in D(B_\lambda^{-s_0})\cap D(\log^{m_0}B_\lambda)$ then for every $m\in \mathbb N$, $m\leq m_0$,
\[
(\log^m B_\lambda)f=\lim_{s\rightarrow 0^+}(-1)^m\partial_s^m B_\lambda^{-s}f,\quad \mbox{ in }L^2(0,\infty).
\]
\end{enumerate}
\end{cor}
We now establish asymptotic expansions for $B_\lambda^s$, $s\in (-1,1)$, $s\neq 0$, through logarithms of $B_\lambda$ in the $L^2(0,\infty)$-setting.

\begin{prop}\label{prop: Prop2.5}
  Let $\lambda >-1$, $s_0\in (0,1)$, $0<s<s_0$, and $m_0\in \mathbb N$.
  \begin{enumerate}
      \item \label{itm: Prop2.5 a} Suppose that $f\in D(B_\lambda^{s_0})\cap D(\log^{m_0+1}B_\lambda)$. Then,
      \[
      B_\lambda^s f=f+\sum_{j=1}^{m_0}\frac{s^j}{j!}(\log^jB_\lambda)f+o(s^{m_0}),\quad \mbox{ in }L^2(0,\infty),
      \]
      and
      \[
      \lim_{s\rightarrow 0^+}\frac{(m_0+1)!}{s^{m_0+1}}\left(B_\lambda^sf-f-\sum_{j=1}^{m_0}\frac{s^j}{j!}(\log^jB_\lambda)f\right)=(\log^{m_0+1}B_\lambda)f,\quad \mbox{ in }L^2(0,\infty).
      \]
      \item \label{itm: Prop2.5 b} Suppose that $f\in D(B_\lambda^{-s_0})\cap D(\log^{m_0+1}B_\lambda)$. Then,
      \[
      B_\lambda^{-s}f=f+\sum_{j=1}^{m_0}\frac{(-1)^js^j}{j!}(\log^jB_\lambda)f+o(s^{m_0}),\quad \mbox{ in }L^2(0,\infty),
      \]
      and
      \[
      \lim_{s\rightarrow 0^+}\frac{(m_0+1)!}{s^{m_0+1}}\left(B_\lambda^{-s}f-f-\sum_{j=1}^{m_0}\frac{(-1)^js^j}{j!}(\log^jB_\lambda)f\right)=(-1)^{m_0+1}(\log^{m_0+1}B_\lambda)f,\quad \mbox{ in }L^2(0,\infty).
      \]
  \end{enumerate}
\end{prop}

\begin{proof}
    We prove \ref{itm: Prop2.5 a}. We have
    \[
     B_\lambda^sf-f-\sum_{j=1}^{m_0}\frac{s^j}{j!}(\log^jB_\lambda)f=h_\lambda \left(\left(y^{2s}
     -1-\sum_{j=1}^{m_0}\frac{s^j}{j!}(2\log y)^j\right)h_\lambda (f)\right),\quad s\in (0,s_0).
    \]
    Since $h_\lambda$ is an isometry on $L^2(0,\infty)$ it follows that
    \[
    \left\|B_\lambda^sf-f-\sum_{j=1}^{m_0}\frac{s^j}{j!}(\log^jB_\lambda)f\right\|_{L^2(0,\infty)}=\left\|\left(y^{2s}
     -1-\sum_{j=1}^{m_0}\frac{s^j}{j!}(2\log y)^j\right)h_\lambda (f)\right\|_{L^2(0,\infty)},\quad s\in (0,s_0).
    \]
 Using the mean value theorem we obtain, for every $s\in (0,s_0)$ and $y\in (0,\infty)$,
    \[
    y^{2s}-1-\sum_{j=1}^{m_0}\frac{s^j}{j!}(2\log y)^j=(2\log y)^{m_0+1}y^{2u}\frac{s^{m_0+1}}{(m_0+1)!},
    \]
     for a certain $u\in (0,s)$. Then,
     \[
     \left|y^{2s}-1-\sum_{j=1}^{m_0}\frac{s^j}{j!}(2\log y)^j\right|\leq C\left(|\log y|^{m_0+1}+y^{2s_0}\right)s^{m_0+1},\quad s\in (0,s_0)\mbox{ and }y\in (0,\infty).
     \]
     We get, for each $s\in (0,s_0)$,
     \[
     \left\|B_\lambda^sf-f-\sum_{j=1}^{m_0}\frac{s^j}{j!}(\log^jB_\lambda)f\right\|_{L^2(0,\infty)}\leq C\left(\|B_\lambda^{s_0}f\|_{L^2(0,\infty)}+\|(\log^{m_0+1}B_\lambda)f\|_{L^2(0,\infty)}\right)s^{m_0+1}.
     \]
     The dominated convergence theorem leads to 
     \begin{align*}
         \lim_{s\rightarrow 0^+}\left\|\frac{(m_0+1)!}{s^{m_0+1}}\left(B_\lambda^sf-f-\sum_{j=1}^{m_0}\frac{s^j}{j!}(\log^jB_\lambda)f\right)-(\log^{m_0+1}B_\lambda)f\right\|_{L^2(0,\infty)}\\
         &\hspace{-10cm}=\lim_{s\rightarrow 0^+}\left\|\left[\frac{(m_0+1)!}{s^{m_0+1}}\left(y^{2s}-1-\sum_{j=1}^{m_0}\frac{(2s)^j}{j!}(\log y)^j\right)-(2\log y)^{m_0+1}\right]h_\lambda (f)\right\|_{L^2(0,\infty)}=0.
     \end{align*}

     The property in \ref{itm: Prop2.5 b} can be established in a similar way.
\end{proof}


\section{Proofs of Theorems \ref{thm: Th1.1} and \ref{thm: Th1.2}}\label{sec: proofs Th1.1 Th1.2}

We now prove Theorems \ref{thm: Th1.1} and \ref{thm: Th1.2}. Actually we will establish a little bit more general representation than those ones stated in Theorems \ref{thm: Th1.1} and \ref{thm: Th1.2}.

According to Proposition \ref{prop: Prop2.1}, $B_\lambda^sf=\mathbb B_{\lambda ,s}f$, where
\[
\mathbb B_{\lambda ,s}f(x)=\frac{1}{\Gamma (-s)}\int_0^\infty \frac{W_t^\lambda (f)(x)-f(x)}{t^{1+s}}dt,\quad a.e.\quad x\in (0,\infty)\mbox{ and }s\in (0,1).
\]
The integral above is absolutely convergent for every $x\in (0,\infty)$ and $s\in (0,1)$. Indeed, we decompose the integral defining $\mathbb B_{\lambda ,s}$ as follows:
\begin{align}\label{eq: decomposition mbbB lambda s}
   \int_0^\infty &\frac{W_t^\lambda (f)(x)-W_t^\lambda (1)(x)f(x)}{t^{1+s}}dt+f(x)\int_0^\infty \frac{W_t^\lambda (1)(x)-1}{t^{1+s}}dt\nonumber\\
    &= \int_{B(x,1)}(f(y)-f(x))\int_0^\infty  \frac{W_t^\lambda (x,y)}{t^{1+s}}dtdy+\int_{(0,\infty)\setminus B(x,1)}f(y)\int_0^\infty \frac{W_t^\lambda (x,y)}{t^{1+s}}dtdy\nonumber\\
    &\quad +f(x)\left[-\int_{(0,\infty)\setminus B(x,1)}\int_0^\infty \frac{W_t^\lambda (x,y)}{t^{1+s}}dtdy+\int_0^\infty \int_{x^2}^\infty \frac{W_t^\lambda (x,y)}{t^{1+s}}dtdy-\int_{x^2}^\infty \frac{dt}{t^{1+s}}\right.\nonumber\\
    &\quad \left.+\left(\int_0^{x/2}\int_0^{x^2}+\int_{2x}^\infty \int_0^{x^2}+\int_{x/2}^{2x}\int_0^{x^2}\right)\frac{W_t^\lambda (x,y)-W_t(x-y)}{t^{1+s}}dtdy\right]\nonumber\\
     &=\int_{B(x,1)}(f(y)-f(x))\int_0^\infty  \frac{W_t^\lambda (x,y)}{t^{1+s}}dtdy+\int_{(0,\infty)\setminus B(x,1)}f(y)\int_0^\infty \frac{W_t^\lambda (x,y)}{t^{1+s}}dtdy\nonumber\\
    &\quad +f(x)\left[-\int_{(0,\infty)\setminus B(x,1)}\int_0^{x^2} \frac{W_t^\lambda (x,y)}{t^{1+s}}dtdy+\int_{B(x,1)}\int_{x^2}^\infty \frac{W_t^\lambda (x,y)}{t^{1+s}}dtdy\right.\nonumber\\
    &\quad \left.+\left(\int_0^{x/2}\int_0^{x^2}+\int_{2x}^\infty \int_0^{x^2}+\int_{x/2}^{2x}\int_0^{x^2}\right)\frac{W_t^\lambda (x,y)-W_t(x-y)}{t^{1+s}}dtdy\right]-\frac{x^{-2s}}{s}f(x)\nonumber\\
    &:=A_1(x,s)+A_2(x,s)+f(x)\sum_{j=3}^7 A_j(x,s)-\frac{x^{-2s}}{s}f(x)
\end{align}
for $x\in (0,\infty)$  and $s\in \left(0,\tfrac{1}{4}\right).$ 

Our objective is to see that each of the above integrals $A_j(x,s)$ is absolutely convergent and defines a smooth function in $s\in (0,\tfrac14)$, for every $x\in (0,\infty)$. We obtain estimates for their derivatives as well.

Let $m\in \mathbb N_0$. We have that
\begin{equation}\label{eq: bound derivative power -1-s}
|\partial_s^m(t^{-1-s})|=|(\log t)^m t^{-1-s}|\leq |\log t|^m\left\{\begin{array}{ll}
t^{-1},&t\geq 1,\\
t^{-1-s},&t\in (0,1),
\end{array}
\right.,\quad \mbox{ and }s\in (0,1).
\end{equation}
We now study the corresponding integrals in \eqref{eq: decomposition mbbB lambda s} by using \eqref{eq: bound derivative power -1-s}.

 According to \eqref{eq: heat Bessel FeynmanKac}, since $f\in {\textup{Lip}}^1(0,\infty)$ we obtain
\begin{align}\label{eq: proof Thm1.1 first term}
    |\partial_s^mA_1(x,s)|&\leq \int_{B(x,1)}|f(y)-f(x)|\int_0^\infty W_t^\lambda (x,y) |\partial_s^m(t^{-1-s})|dtdy\nonumber\\
    &\leq C\int_{B(x,1)}|y-x|\left(\int_0^1\frac{e^{-c\frac{|x-y|^2}{t}}}{t^{s+3/2}}|\log t|^m dt+\int_1^\infty \frac{e^{-c\frac{|x-y|^2}{t}}}{t^{3/2}}|\log t|^m dt\right)dy\nonumber\\
    &\leq C\int_{B(x,1)}|y-x| \left(\int_0^\infty\frac{e^{-c\frac{|x-y|^2}{t}}}{t^{7 /4+\varepsilon }}dt+\int_1^\infty \frac{dt}{t^{3/2-\varepsilon}}\right)dy\nonumber\\
    &\leq C\int_{B(x,1)}|y-x|\left(\frac{1}{|x-y|^{3 /2+2\varepsilon}}+1\right)dy\leq C,\quad x\in (0,\infty)\mbox{ and }s\in \left(0,\tfrac{1}{4}\right),
\end{align}
where $\varepsilon \in (0,\tfrac14)$.

Since $f\in C_c^\infty(0,\infty)$, \eqref{eq: heat Bessel FeynmanKac} leads to
\begin{align}\label{eq: proof Thm1.1 second term}
    |\partial_s^mA_2(x,s)|&\leq \int_{(0,\infty) \setminus B(x,1)}|f(y)|\int_0^\infty W_t^\lambda (x,y) |\partial_s^m(t^{-1-s})|dtdy\nonumber\\
    &\leq C\int_{(0,\infty) \setminus B(x,1)}|f(y)|\left(\int_0^1\frac{e^{-c\frac{|x-y|^2}{t}}}{t^{s+3/2+\varepsilon }}dt+\int_1^\infty \frac{e^{-c\frac{|x-y|^2}{t}}}{t^{3/2-\varepsilon}}dt\right)dy\nonumber\\
    &\leq C\int_{(0,\infty) \setminus B(x,1)}|f(y)| \left(\int_0^\infty\frac{e^{-c\frac{|x-y|^2}{t}}}{t^{2+\varepsilon }}dt+\int_1^\infty \frac{dt}{t^{3/2-\varepsilon}}\right)dy\nonumber\\
    &\leq C\int_{(0,\infty) \setminus B(x,1)}|f(y)|\left(\frac{1}{|x-y|^{2+2\varepsilon}}+1\right)dy\nonumber\\
    &\leq C\|f\|_{L^1(0,\infty)},\quad x\in (0,\infty)\mbox{ and }s\in (0,\tfrac12),
\end{align}
with $\varepsilon \in (0,\tfrac12)$.

 By using \eqref{eq: heat Bessel FeynmanKac} we get
\begin{align}\label{eq: proof Thm1.1 third term}
    |\partial_s^mA_3(x,s)|&\leq \int_{(0,\infty) \setminus B(x,1)}\int_0^{x^2} W_t^\lambda (x,y) |\partial_s^m (t^{-1-s})|dtdy\nonumber\\
    &\leq C\int_{(0,\infty) \setminus B(x,1)}\left(\int_0^1\frac{e^{-c\frac{|x-y|^2}{t}}}{t^{s+3/2+\varepsilon }}dt+\mathcal X_{(1,\infty)}(x)\int_1^{x^2}\frac{e^{-c\frac{|x-y|^2}{t}}}{t^{3/2-\varepsilon}}dt\right)dy\nonumber\\
    &\leq C\int_{(0,\infty) \setminus B(x,1)} \left(\int_0^1\frac{e^{-c\frac{|x-y|^2}{t}}}{t^{2+\varepsilon }}dt+\mathcal X_{(1,\infty)}(x)e^{-c\frac{|x-y|^2}{x^2}}\int_1^\infty \frac{dt}{t^{3/2-\varepsilon}}\right)dy\nonumber\\
    &\leq C\int_{(0,\infty) \setminus B(x,1)}\left(\frac{1}{|x-y|^{2+2\varepsilon}}+e^{-c\frac{|x-y|^2}{x^2}}\right)dy\nonumber\\
    &\leq C\int_1^\infty \left(\frac{1}{\rho ^{2+2\varepsilon}}+e^{-c\frac{\rho ^2}{x^2}}\right)d\rho\nonumber\\
    &\leq C(1+x),\quad x\in (0,\infty)\mbox{ and }s\in (0,\tfrac12),
\end{align}
with $\varepsilon \in (0,\tfrac12)$.

  Equation \eqref{eq: heat Bessel FeynmanKac} implies that
\begin{align}\label{eq: proof Thm1.1 fourth term}
    |\partial_s^mA_4(x,s)|&\leq \int_{B(x,1)}\int_{x^2}^\infty W_t^\lambda (x,y) |\partial_s^m(t^{-1-s})|dtdy\nonumber\\
    &\leq C\int_{B(x,1)}\left(\int_1^\infty\frac{e^{-c\frac{|x-y|^2}{t}}}{t^{3/2-\varepsilon }}dt+\mathcal X_{(0,1)}(x)\int_{x^2}^1\frac{e^{-c\frac{|x-y|^2}{t}}}{t^{s+3/2+\varepsilon}}dt\right)dy\nonumber\\
    &\leq C\int_{B(x,1)} \left(\int_1^\infty \frac{dt}{t^{3/2-\varepsilon }}dt+\mathcal X_{(0,1)}(x)\int_{x^2}^\infty \frac{dt}{t^{2+\varepsilon}}\right)dy\nonumber\\
    &\leq C\left(1+x^{-2(1+\varepsilon)}\right),\quad x\in (0,\infty)\mbox{ and }s\in (0,\tfrac12),
\end{align}
with $\varepsilon \in (0,\tfrac12)$.

 By using \eqref{eq: heat Bessel FeynmanKac} we obtain
\begin{align}\label{eq: proof Thm1.1 fifth term}
    |\partial_s^mA_5(x,s)|& \leq \int_0^{x/2}\int_0^{x^2} |W_t^\lambda (x,y) -W_t(x-y)||\partial_s^m (t^{-1-s})|dtdy\nonumber\\
    &\leq C\int_0^{x/2}\left(\int_0^1\frac{e^{-c\frac{|x-y|^2}{t}}}{t^{s+3/2+\varepsilon }}dt+\mathcal X_{(1,\infty)}(x)\int_1^{x^2}\frac{e^{-c\frac{|x-y|^2}{t}}}{t^{3/2-\varepsilon}}dt\right)dy\nonumber\\
    &\leq C \int_0^{x/2} \left(\int_0^1\frac{e^{-c\frac{x^2}{t}}}{t^{2+\varepsilon }}dt+\int_1^\infty \frac{dt}{t^{3/2-\varepsilon}}\right)dy\nonumber\\
    &\leq C\int_0^{x/2}\left(\frac{1}{x^{2+2\varepsilon}}+1\right)dy\nonumber\\
    &\leq C\left(\frac{1}{x^{1+2\varepsilon}}+x\right),\quad x\in (0,\infty)\mbox{ and }s\in (0,\tfrac12),
\end{align}
with $\varepsilon\in (0,\tfrac12)$.

 Inequality  \eqref{eq: heat Bessel FeynmanKac} leads to
\begin{align}\label{eq: proof Thm1.1 sixth term}
    |\partial_s^mA_6(x,s)|&\leq \int_{2x}^\infty \int_0^{x^2} |W_t^\lambda (x,y) -W_t(x-y)||\partial_s^m(t^{-1-s})|dtdy\nonumber\\
    &\leq C\int_{2x}^\infty \left(\int_0^1\frac{e^{-c\frac{|x-y|^2}{t}}}{t^{s+3/2+\varepsilon }}dt+\mathcal X_{(1,\infty)}(x)\int_1^{x^2}\frac{e^{-c\frac{|x-y|^2}{t}}}{t^{3/2-\varepsilon}}dt\right)dy\nonumber\\
    &\leq C\int_{2x}^\infty \left(\int_0^\infty\frac{e^{-c\frac{y^2}{t}}}{t^{2+\varepsilon }}dt+e^{-c\frac{y^2}{x^2}}\right)dy\nonumber\\
    &\leq C\int_{2x}^\infty \left(\frac{1}{y^{2+2\varepsilon}}+e^{-c\frac{y^2}{x^2}}\right)dy\nonumber\\
    &\leq C\left(\frac{1}{x^{1+2\varepsilon}}+x\right),\quad x\in (0,\infty)\mbox{ and }s\in (0,\tfrac12),
\end{align}
with $\varepsilon \in (0,\tfrac12)$.

 By using \eqref{eq: heat Bessel FeynmanKac} and (\ref{eq: heat Bessel comparison t<xy}) we obtain
\begin{align}\label{eq: proof Thm1.1 last term}
    |\partial_s^mA_7(x,s)|&\leq \int_{x/2}^{2x}\int_0^{x^2} |W_t^\lambda (x,y) -W_t(x-y)||\partial_s^m (t^{-1-s})|dtdy\nonumber\\
    &\leq C\left[\int_{x/2}^{2x}\Bigg(\int_0^1\left(\frac{e^{-c\frac{x^2+y^2}{t}}}{\sqrt{t}}+\frac{\sqrt{t}}{xy}e^{-c\frac{|x-y|^2}{t}}\right)\frac{|\log t|^m}{t^{s+1}}dt\right.\nonumber\\
    &\quad +\left.\mathcal X_{(1,\infty)}(x)\int_1^{x^2}\left(\frac{e^{-c\frac{x^2+y^2}{t}}}{\sqrt{t}}dt+\frac{\sqrt{t}}{xy}e^{-c\frac{|x-y|^2}{t}}\right)\frac{|\log t|^m}{t}dt\Bigg)dy\right]\nonumber\\
    &\leq C\left[\int_{x/2}^{2x}\Bigg(\int_0^1\left(\frac{e^{-c\frac{x^2+y^2}{t}}}{t^{2+\varepsilon}}+\frac{1}{xyt^{3/4+\varepsilon}}e^{-c\frac{|x-y|^2}{t}}\right)dt\right.\nonumber\\
    &\quad +\chi_{(1,\infty)}(x)\left.\int_1^{x^2}\frac{1}{t^{1/2-\varepsilon}}\left(\frac{1}{t}+\frac{1}{xy}\right)dt\Bigg)dy\right]\nonumber\\
    &\leq C\int_{x/2}^{2x}\frac{1}{(x^2+y^2)^{1+\varepsilon}}+\frac{1}{xy}+1+x^{2\epsilon-1}dy\leq C\left(\frac{1}{x^{1+2\varepsilon}}+\frac{1}{x}+x+x^{2\epsilon}\right)\nonumber\\
    &\leq C\left(x+\frac{1}{x^{1+2\varepsilon}}\right),\quad x\in (0,\infty)\mbox{ and }s\in \left(0,\tfrac{1}{4}\right),
\end{align}
with $\varepsilon \in (0,\tfrac14)$.

Note that the constants $C>0$ in the above estimates do not depend on $x$ and $s$.

We now consider the function 
\[
F(t,s)=\frac{s}{\Gamma (1-s)}t^{-1-s},\quad t\in (0,\infty)\mbox{ and }s\in \left(0,\tfrac{1}{4}\right).
\]
Its derivatives can be written as follows
\begin{align*}
    \partial_s^m F(t,s)&=s\sum_{j=0}^m\binom{m}{j}\partial_s^{m-j}\left(\frac{1}{\Gamma (1-s)}\right)\partial_s^j(t^{-1-s})\\
    &\quad +m\sum_{j=0}^{m-1}\binom{m-1}{j}\partial_s^{m-1-j}\left(\frac{1}{\Gamma (1-s)}\right)\partial_s^j(t^{-1-s}),\quad t\in (0,\infty)\mbox{ and }s\in \left(0,\tfrac{1}{4}\right).
\end{align*}

Estimates \eqref{eq: proof Thm1.1 first term}-\eqref{eq: proof Thm1.1 last term}  justify differentiation under the integral sign. We get the following expression by means of $\partial_s^m F$
\begin{align*}
    \partial_s^m B_\lambda^sf(x)&=-\Bigg(\int_{B(x,1)}(f(y)-f(x))\int_0^\infty W_t^\lambda (x,y)\partial_s^m F(t,s)dtdy\\
    &\quad +\int_{(0,\infty)\setminus B(x,1)}f(y)\int_0^\infty W_t^\lambda (x,y)\partial_s^m F(t,s)dtdy\\
    &\quad +f(x)\left[-\int_{(0,\infty)\setminus B(x,1)}\int_0^{x^2}W_t^\lambda (x,y)\partial_s^m F(t,s)dtdy\right.\\
    &\quad +\int_{B(x,1)}\int_{x^2}^\infty W_t^\lambda (x,y)\partial_s^m F(t,s)dtdy\\
    &\quad \left.+\Bigg(\int_0^{x/2}\int_0^{x^2}+\int_{2x}^\infty \int_0^{x^2}+\int_{x/2}^{2x}\int_0^{x^2}\Bigg)(W_t^\lambda (x,y)-W_t(x-y))\partial_s^m F(t,s)dtdy\right]\\
    &\quad +\partial_s^m \Big(\frac{x^2}{s}F(x^2,s)\Big)f(x)\Bigg),\quad a.e. \quad x\in (0,\infty)\mbox{ and }s\in \left(0,\tfrac{1}{4}\right).
\end{align*}

In particular, we have that 
\begin{align*}
    B_\lambda^sf(x)&=-\frac{s}{\Gamma(1-s)}\Bigg[\int_{B(x,1)}(f(y)-f(x))\int_0^\infty \frac{W_t^\lambda (x,y)}{t^{1+s}}dtdy\\
    &\quad +\int_{(0,\infty)\setminus B(x,1)}f(y)\int_0^\infty \frac{W_t^\lambda (x,y)}{t^{1+s}}dtdy +f(x)\left(-\int_{(0,\infty)\setminus B(x,1)}\int_0^{x^2}\frac{W_t^\lambda (x,y)}{t^{1+s}}dtdy\right. \\
    &\quad +\int_{B(x,1)}\int_{x^2}^\infty \frac{W_t^\lambda (x,y)}{t^{1+s}}dtdy\left.+\int_0^\infty \int_0^{x^2}\frac{W_t^\lambda (x,y)-W_t(x-y)}{t^{1+s}}dtdy\right)\Bigg]\\
    &\quad +\frac{x^{-2s}}{\Gamma(1-s)}f(x),\quad a.e.\ x\in (0,\infty)\mbox{ and }s\in \left(0,\tfrac{1}{4}\right).
\end{align*}
Also, by breaking the $t$-integral $\int_0^\infty dt$ into the integrals $\int_0^1dt$ and $\int_1^\infty$ instead of $\int_0^{x^2}dt$ and $\int_{x^2}^\infty dt$ we obtain
\begin{align*}
    B_\lambda^sf(x)&=-\frac{s}{\Gamma(1-s)}\Bigg(\int_{B(x,1)}(f(y)-f(x))\int_0^\infty \frac{W_t^\lambda (x,y)}{t^{1+s}}dtdy\\
     &\quad +\int_{(0,\infty)\setminus B(x,1)}f(y)\int_0^\infty \frac{W_t^\lambda (x,y)}{t^{1+s}}dtdy\\
    &\quad +f(x)\left[-\int_{(0,\infty)\setminus B(x,1)}\int_0^{1}\frac{W_t^\lambda (x,y)}{t^{1+s}}dtdy\right.\\
    &\quad +\int_{B(x,1)}\int_{1}^\infty \frac{W_t^\lambda (x,y)}{t^{1+s}}dtdy\\
    &\quad \left.+\int_0^\infty \int_0^{1}\frac{W_t^\lambda (x,y)-W_t(x-y)}{t^{1+s}}dtdy\right]\Bigg)\\
    &\quad + \frac{f(x)}{\Gamma(1-s)},\quad a.e.\quad x\in (0,\infty)\mbox{ and }s\in \left(0,\tfrac{1}{4}\right),
\end{align*}
and property \ref{itm: Th1.1 a} in Theorem~\ref{thm: Th1.1} is proved.

According to Corollary \ref{cor: Cor2.5} we can find a sequence $\{s_\ell\}_{\ell\in \mathbb{N}}\subset (0,\infty)$ such that $\lim_{\ell\to \infty}s_\ell=0$ and $(\log^m B_\lambda)(f)(x)=\lim_{\ell\to\infty}\partial_{s}^m B_\lambda^s(f)(x)\Big|_{s=s_\ell}$, for a.e. $x\in (0,\infty).$

Estimates \eqref{eq: proof Thm1.1 first term}-\eqref{eq: proof Thm1.1 last term} also allow us to take limit as $s\rightarrow 0^+$ inside the integral and we obtain
\begin{align*}
    (\log^m B_\lambda)f(x)&=\lim_{\ell\rightarrow \infty}\partial_s^m B_\lambda^sf(x)\Big|_{s=s_\ell}=-m\sum_{j=0}^{m-1}\binom{m-1}{j}\partial_s^{m-1-j}\left(\frac{1}{\Gamma (1-s)}\right)\Bigg|_{s=0}(-1)^j\\
    &\quad \times \Bigg(\int_{B(x,1)}(f(y)-f(x))\int_0^\infty W_t^\lambda (x,y)\frac{(\log t)^j}{t}dtdy\\
    &\quad -\int_{(0,\infty)\setminus B(x,1)}f(y)\int_0^\infty W_t^\lambda (x,y)\frac{(\log t)^j}{t}dtdy\\
    &\quad +f(x)\left[-\int_{(0,\infty)\setminus B(x,1)}\int_0^{x^2}W_t^\lambda (x,y)\frac{(\log t)^j}{t}dtdy\right.\\
    &\quad +\int_{B(x,1)}\int_{x^2}^\infty W_t^\lambda (x,y)\frac{(\log t)^j}{t}dtdy\\
    &\quad \left.+\left(\int_0^{x/2}\int_0^{x^2}+\int_{2x}^\infty \int_0^{x^2}+\int_{x/2}^{2x}\int_0^{x^2}\right)(W_t^\lambda (x,y)-W_t(x-y))\frac{(\log t)^j}{t}dtdy\right]\Bigg)\\
    &\quad -\sum_{j=0}^m\binom{m}{j}\partial_s^{m-j}\left(\frac{1}{\Gamma (1-s)}\right)\Bigg|_{s=0}(-2)^j(\log x)^jf(x),\quad a.e. \quad  x\in (0,\infty).
\end{align*}
In particular, we have that for almost every $x\in (0,\infty)$,
\begin{align*}
    (\log B_\lambda)f(x)&=-\left(\int_{B(x,1)}(f(y)-f(x))\int_0^\infty \frac{W_t^\lambda (x,y)}{t}dtdy\right.\\
    &\quad -\int_{(0,\infty)\setminus B(x,1)}f(y)\int_0^\infty \frac{W_t^\lambda (x,y)}{t}dtdy\\
    &\quad +f(x)\left[-\int_{(0,\infty)\setminus B(x,1)}\int_0^{x^2}\frac{W_t^\lambda (x,y)}{t}dtdy+\int_{B(x,1)}\int_{x^2}^\infty \frac{W_t^\lambda (x,y)}{t}dtdy\right.\\
    &\quad \left.\left.+\int_0^\infty\int_0^{x^2}(W_t^\lambda (x,y)-W_t(x-y))\frac{dt}{t}dtdy\right]\right)-\Gamma '(1)f(x)-2(\log x )f(x).
\end{align*}
Thus, we obtain a representation of $(\log B_\lambda)f$ different from the one established in \cite[Theorem 2.4]{FLZ}. As it can be noted, our proof works for every $m\in \mathbb N$ and it is shorter than the one in \cite[Theorem 2.4]{FLZ}. Also, the representation in \cite[Theorem 2.4]{FLZ} for $(\log B_\lambda)f$ appears when the $t$-integrals are decomposed in $\int_0^1+\int_1^\infty$ instead of in $\int_0^{x^2}+\int_{x^2}^\infty$. For $(\log^m B_\lambda)f$, we can also obtain the following representation:
\begin{align}\label{eq: representation log m Bessel s=0}
    (\log^m B_\lambda)f(x)&=-m\sum_{j=0}^{m-1}\binom{m-1}{j}\partial_s^{m-1-j}\left(\frac{1}{\Gamma (1-s)}\right)\Bigg|_{s=0}(-1)^j\nonumber\\
    &\quad \times \Bigg(\int_{B(x,1)}(f(y)-f(x))\int_0^\infty W_t^\lambda (x,y)\frac{(\log t)^j}{t}dtdy\nonumber\\
    &\quad +\int_{(0,\infty)\setminus B(x,1)}f(y)\int_0^\infty W_t^\lambda (x,y)\frac{(\log t)^j}{t}dtdy\nonumber\\
    &\quad +f(x)\left[-\int_{(0,\infty)\setminus B(x,1)}\int_0^1W_t^\lambda (x,y)\frac{(\log t)^j}{j}dtdy\right.\nonumber\\
    &\quad +\int_{B(x,1)}\int_1^\infty W_t^\lambda (x,y)\frac{(\log t)^j}{t}dtdy\nonumber\\
    &\quad \left.+\int_0^\infty\int_0^1(W_t^\lambda (x,y)-W_t(x-y))\frac{(\log t)^j}{t}dtdy\right]\Bigg)\nonumber\\
    &\quad +\partial_s^{m}\left(\frac{1}{\Gamma (1-s)}\right)\Bigg|_{s=0}f(x),\quad a.e. \quad  x\in (0,\infty).
    \end{align}
We now consider the negative power of $B_\lambda$. We have that
\[
B_\lambda^{-s}f(x)=\frac{1}{\Gamma (s)}\int_0^\infty W_t^\lambda (f)(x)t^{s-1}dt,\quad a.e.\quad  x\in (0,\infty)\mbox{ and }s\in (0,1).
\]
We decompose $B_\lambda^{-s}$ as follows
\begin{align}\label{eq: decomposition Bessel negative power}
   B_\lambda^{-s}f(x)&=\frac{s}{\Gamma (1+s)}\left(\int_0^\infty (W_t^\lambda (f)(x)-W_t^\lambda (1)(x)f(x))t^{s-1}dt\right.\nonumber\\
    &\quad +\left.f(x)\left[\int_0^1(W_t^\lambda (1)(x)-W_t(1)(x))t^{s-1}dt +\int_1^\infty  W_t^\lambda (1)(x)t^{s-1}dt+\frac{1}{s}\right]\right)\nonumber\\
    &=\frac{s}{\Gamma (1+s)}\left(\int_{B(x,1)} (f(y)-f(x))\int_0^\infty W_t^\lambda (x,y)t^{s-1}dtdy\right.\nonumber\\
    &\quad +\int_{(0,\infty)\setminus B(x,1)}f(y)\int_0^\infty W_t^\lambda (x,y)t^{s-1}dtdy\nonumber\\
    &\quad +f(x)\left[\int_0^\infty \int_0^1(W_t^\lambda(x,y)-W_t(x-y))t^{s-1}dtdy\right.\nonumber\\
    &\quad -\int_{(0,\infty)\setminus B(x,1)}\int_0^1 W_t^\lambda (x,y)t^{s-1}dtdy\nonumber\\
    &\quad \left.\left.+\int_{B(x,1)}\int _1^\infty W_t^\lambda (x,y)t^{s-1}dtdy\right]\right)\nonumber\\
    &\quad +\frac{f(x)}{\Gamma (1+s)},\quad a.e.\ x\in (0,\infty)\mbox{ and }s\in (0,1).
\end{align}
Let $m\in \mathbb N$. We have that
\[
|\partial_s^m (t^{s-1})|=|(\log t)^m t^{s-1}|\leq |\log t|^m\left\{\begin{array}{ll}
t^{s-1},&t\geq 1,\\
t^{-1},&t\in (0,1),
\end{array}
\right.,\quad \mbox{ and }s\in (0,1).
\]
By proceeding as above we can see that
\begin{align*}
    \partial_s^m B_\lambda^{-s}f(x)&=-\left(\int_{B(x,1)}(f(y)-f(x))\int_0^\infty W_t^\lambda (x,y)\partial_s^m F(t,-s)dtdy\right.\\
    &\quad +\int_{(0,\infty)\setminus B(x,1)}f(y)\int_0^\infty W_t^\lambda (x,y)\partial_s^m F(t,-s)dtdy\\
    &\quad +f(x)\left[\int_0^\infty\int_0^1(W_t^\lambda (x,y)-W_t(x-y))\partial_s^m F(t,-s)dtdy\right.\\
    &\quad -\int_{(0, \infty)\setminus B(x,1)}\int_0^1 W_t^\lambda (x,y)\partial_s^m F(t,-s)dtdy\\
    &\quad \left.+\int_{B(x,1)}\int_1^\infty W_t^\lambda (x,y)\partial_s^m F(t,-s)dtdy\right]\\
    &\quad +\partial_s^m \left(\frac{1}{\Gamma (1+s)}\right)f(x),\quad a.e.\quad  x\in (0,\infty)\mbox{ and }s\in \left(0,\tfrac{1}{4}\right),
\end{align*}
and
\begin{equation}\label{eq: representation log^m negative s}
(\log^m B_\lambda)f(x)=(-1)^m\partial_s^m B_\lambda^{-s}(f)(x)\Big|_{s=0^+},\quad a.e.\quad x\in (0,\infty).    
\end{equation}

\begin{rem}\label{rem: log}
Suppose that $x\in (0,\infty)$ and  $f\in L^1\big((0,\infty),\frac{dy}{1+y}\big)\cap {\textup{Lip}}_{{\rm loc},x}(0,\infty)$, for a certain $\theta\in (0,1]$. We are going to see that the integrals in \ref{itm: Th1.1 a} and \ref{itm: Th1.1 b} of Theorem~\ref{thm: Th1.1} are absolutely convergent for every $m\in\mathbb{N}$ when $s\in (0,\tfrac{\theta}{4})$. Let $m\in \mathbb{N}$. By proceeding as in \eqref{eq: proof Thm1.1 first term} we obtain
\[
\int_{B(x,1)}|f(y)-f(x)|\int_0^\infty W_t^\lambda(x,y)|\partial_s^m (t^{-1-s})|dtdy\le C\|f\|_{{\textup{Lip}}_{{\rm loc},x}(0,\infty)},\quad s\in (0,\tfrac{\theta}{4}).
\]
By using (\ref{eq: heat Bessel FeynmanKac}) and (\ref{eq: heat Bessel t>xy}) we get with $0<\epsilon <\frac{\lambda}{2}+\frac{1}{4}$
\begin{align*}
&\int_0^\infty W_t^\lambda(x,y)\frac{|\log t|^m}{t^{1+s}}dt\\
&\le C\left(\int_0^1 e^{-c\frac{(x-y)^2}{t}}\frac{|\log t|^m}{t^{\frac{3}{2}+s}}dt+\int_1^{xy}e^{-c\frac{(x-y)^2}{t}}\frac{|\log t|^m}{t^{\frac{3}{2}}}dt+\int_{xy}^\infty e^{-c\frac{x^2+y^2}{t}}\frac{(xy)^{\lambda+\frac{1}{2}}|\log t|^m}{t^{\lambda+2}}dt\right)\\
&\le C\left(\int_0^1 \frac{e^{-c\frac{(x-y)^2}{t}}}{t^{2+\epsilon}}dt+\int_1^{xy}\frac{e^{-c\frac{(x-y)^2}{t}}}{t^{\frac{3}{2}-\epsilon}}dt+\int_{xy}^\infty\frac{ e^{-c\frac{x^2+y^2}{t}}}{t^{\lambda+2-\epsilon}}dt(xy)^{\lambda+\frac{1}{2}}\right)\\
&\le C\left(\frac{1}{|x-y|^{2+2\epsilon}}+\chi_{(2x,+\infty)}(y)\int_1^{xy}\frac{e^{-cy^2/xy}}{t^{\frac{3}{2}-\epsilon}}dt+\chi_{(0,2x)}(y)\int_1^\infty t^{\epsilon-\frac{1}{2}}dt+\frac{(xy)^{\lambda+\frac{1}{2}}}{(x+y)^{2\lambda+2-2\epsilon}}\right)\\
&\le C\left(\frac{1}{|x-y|^{2+2\epsilon}}+\chi_{(2x,+\infty)}(y){e^{-cy/x}}+\chi_{(0,2x)}(y)+\frac{x^{\lambda+\frac{1}{2}}}{(x+y)^{\lambda+\frac{3}{2}-2\epsilon}}\right),
\end{align*}
for every $y\in (0,\infty)$, $xy>1$ and $s\in (0,\frac{1}{2})$, and
\begin{align*}
&\int_0^\infty W_t^\lambda(x,y)\frac{|\log t|^m}{t^{1+s}}dt\\
&\le C\Bigg(\int_0^{xy} e^{-c\frac{(x-y)^2}{t}}\frac{|\log t|^m}{t^{\frac{3}{2}+s}}dt+\int_{xy}^{1}e^{-c\frac{x^2+y^2}{t}}\frac{|\log t|^m}{t^{\lambda+\frac{5}{2}}}dt(xy)^{\lambda+\frac{1}{2}}\\
&\quad +\int_{1}^\infty e^{-c\frac{x^2+y^2}{t}}\frac{(xy)^{\lambda+\frac{1}{2}}|\log t|^m}{t^{\lambda+2}}dt\Bigg)\\
&\le C\left(\int_0^{xy} \frac{e^{-c\frac{(x-y)^2}{t}}}{t^{2+\epsilon}}dt+\int_{xy}^1\frac{e^{-c\frac{x^2+y^2}{t}}}{t^{\lambda+\frac{5}{2}+\epsilon}}dt(xy)^{\lambda+\frac{1}{2}}+\int_{1}^\infty\frac{ e^{-c\frac{x^2+y^2}{t}}}{t^{\lambda+2-\epsilon}}dt(xy)^{\lambda+\frac{1}{2}}\right)\\
&\le C\left(\frac{1}{|x-y|^{2+2\epsilon}}+\frac{x^{\lambda+\frac{1}{2}}}{(x+y)^{\lambda+\frac{5}{2}+2\epsilon}}+\frac{x^{\lambda+1/2}}{(x+y)^{\lambda+\frac{3}{2}-2\epsilon}}\right),
\end{align*}
for every $y\in (0,\infty)$, $xy<1$ and $s\in (0,\frac{1}{2})$.

We have that
\begin{align*}
&\int_{(0,\infty)\setminus B(x,1)}|f(y)|\Bigg(\frac{1}{|x-y|^{2+2\epsilon}}+\frac{x^{\lambda+\frac{1}{2}}}{(x+y)^{\lambda+\frac{3}{2}-2\epsilon}}+\frac{x^{\lambda+\frac{1}{2}}}{(x+y)^{\lambda+\frac{3}{2}+2\epsilon}}\\
&\quad+\chi_{(2x,+\infty)}(y)e^{-cy/x}+\chi_{(0,2x)}(y)\Bigg)dy\\
&\le C\int_{(0,\infty)\setminus B(x,1)}|f(y)|\Bigg(\chi_{(0,2x)}(y)(1+x^{\lambda+\frac{1}{2}})+\chi_{(2x,+\infty)}(y)\frac{1+x^{2\epsilon}+x^{-1-2\epsilon}+x}{|x-y|}\Bigg)dy\\
&\le C\int_0^\infty\frac{|f(y)|}{1+y}dy\Big((1+x)(1+x^{\lambda+\frac{1}{2}})+1+x^{2\epsilon}+x^{-1-2\epsilon}+x\Big).
\end{align*}
Here $C>0$ does not depend on $x$.

Also by using dominated convergence as in the proofs of Theorems~\ref{thm: Th1.1} and \ref{thm: Th1.2} we obtain that
\begin{align*}
    \lim_{s\to 0^+}\partial_s^m \mathbb{B}_{\lambda,s}f(x)&=-m\sum_{j=0}^{m-1}\binom{m-1}{j}\partial_s^{m-1-j}\left(\frac{1}{\Gamma (1-s)}\right)\Bigg|_{s=0}(-1)^j\\
    &\quad \times \left(\int_{B(x,1)}(f(y)-f(x))\int_0^\infty W_t^\lambda (x,y)\frac{(\log t)^j}{t}dtdy\right.\\
    &\quad +\int_{(0,\infty)\setminus B(x,1)}f(y)\int_0^\infty W_t^\lambda (x,y)\frac{(\log t)^j}{t}dtdy\\
    &\quad +f(x)\left[-\int_{(0,\infty)\setminus B(x,1)}\int_0^1W_t^\lambda (x,y)\frac{(\log t)^j}{j}dtdy\right.\\
    &\quad +\int_{B(x,1)}\int_1^\infty W_t^\lambda (x,y)\frac{(\log t)^j}{j}dtdy\\
    &\quad \left.\left.+\int_0^\infty\int_0^1(W_t^\lambda (x,y)-W_t(x-y))\frac{(\log t)^j}{t}dtdy\right]\right)\\
    &\quad +\partial_s^{m}\left(\frac{1}{\Gamma (1-s)}\right)\Bigg|_{s=0}f(x),\quad a.e. \quad  x\in (0,\infty).
    \end{align*}

As it was mentioned this property justifies to define
\[
\log^m B_\lambda f(x)=\lim_{s\to 0^+}\partial_s^m \mathbb{B}_{\lambda,s}f(x),
\]
when $x\in (0,\infty)$, $m\in \mathbb{N}$ and $f\in L^1((0,\infty),\frac{dy}{1+y})\cap {\textup{Lip}}_{{\rm loc},x}(0,\infty)$, for a certain $\theta\in (0,1]$.

We also have that
\[
\lim_{s\to 0^+}\partial_s^m \mathbb{B}_{\lambda,s}f(x)=\lim_{s\to 0^+}(-1)^m\partial_s^m \mathbb{B}_{\lambda,-s}f(x),
\]
provided that $x\in (0,\infty)$, $m\in \mathbb{N}$ and $f\in L^1\big((0,\infty),\frac{dy}{1+y}\big)\cap {\textup{Lip}}_{{\rm loc},x}(0,\infty)$, for a certain $\theta\in (0,1]$.
\end{rem}


\section{Proof of Theorem~\ref{thm: Th1.3}}\label{sec: proof Th1.3}

We consider again the function
\[
F(t,s)=\frac{s}{\Gamma (1-s)}t^{-1-s},\quad t\in (0,\infty)\mbox{ and }s\in (0,1).
\]
Let $t\in (0,\infty)$, $s\in (0,1)$ and $m\in \mathbb N$. By using the mean value theorem we obtain, for certain $u\in (0,s)$,
\begin{align*}
    F(t,s)&=\sum_{j=1}^m\frac{\partial_v^jF(t,v)}{j!}\Bigg|_{v=0}s^j+\frac{\partial_v^{m+1}F(t,v)}{(m+1)!}\Bigg|_{v=u}s^{m+1}\\
    &=\sum_{j=1}^m\frac{s^j}{j!}j\sum_{k=0}^{j-1}\binom{j-1}{k}\partial_v^{j-1-k}\left(\frac{1}{\Gamma (1-v)}\right)\Bigg|_{v=0}(-1)^k\frac{(\log t)^k}{t}\\
    &\quad +\frac{s^{m+1}}{(m+1)!}\left(u\sum_{k=0}^{m+1}\binom{m+1}{k}
    \partial_v^{m+1-k}\left(\frac{1}{\Gamma (1-v)}\right)\Bigg|_{v=u}(-1)^k\frac{(\log t)^k}{t^{u+1}}\right.\\
    &\quad +\left.(m+1)\sum_{k=0}^m\binom{m}{k}\partial_v^{m-k}\left(\frac{1}{\Gamma (1-v)}\right)\Bigg|_{v=u}(-1)^k\frac{(\log t)^k}{t^{u+1}}\right).
\end{align*}
Then, for $s\in (0,\tfrac12)$,
\begin{align*}
    \left|F(t,s)-\sum_{j=1}^m\frac{s^j}{j!}j\sum_{k=0}^{j-1}\binom{j-1}{k}\partial_v^{j-1-k}\left(\frac{1}{\Gamma (1-v)}\right)\Bigg|_{v=0}(-1)^k\frac{(\log t)^k}{t}\right|&\\
    &\hspace{-3cm}\leq Cs^{m+1}(1+|\log t|^{m+1})\left\{\begin{array}{ll}
    t^{-1-s},&t\in (0,1),\\
    t^{-1},&t\in [1,\infty).
    \end{array}
    \right.
    \end{align*}
Here, $C>0$ does not depend on $(t,s)\in (0,\infty)\times (0,\tfrac12)$.

We also have that
\[
\frac{1}{\Gamma (1-s)}=1+\sum_{j=1}^m\frac{s^j}{j!}\partial_v^{j}\left(\frac{1}{\Gamma (1-v)}\right)\Bigg|_{v=0}+\frac{s^{m+1}}{(m+1)!}\partial_v^{m+1}\left(\frac{1}{\Gamma (1-v)}\right)\Bigg|_{v=u},
\]
for certain $u\in (0,s)$.

From \eqref{eq: decomposition mbbB lambda s} (by changing $\int_0^{x^2}$ and $\int_{x^2}^\infty$ by $\int_0^1$ and $\int_1^\infty$, respectively) and \eqref{eq: representation log m Bessel s=0} we deduce that, for every $x\in (0,\infty)$ and $s\in (0,\tfrac{\theta}{4})$,
\begin{align}\label{eq: estimate limit Blambda s positive}
\left|B_\lambda^sf(x)-\left(f(x)+\sum_{j=1}^m\frac{s^j}{j!}(\log^jB_\lambda)(f)(x)\right)\right|&\leq |f(x)|\left|\frac{1}{\Gamma (1-s)}-1-\sum_{j=1}^m\frac{s^j}{j!}\partial_v^{j}\left(\frac{1}{\Gamma (1-v)}\right)\Bigg|_{v=0}\right|\nonumber\\
&\hspace{-5cm}\quad +\int_{B(x,1)}|f(x)-f(y)|\int_0^\infty W_t^\lambda (x,y)\left|F(t,s)-\sum_{j=1}^m\frac{s^j}{j!}\partial_v^jF(t,v)\Bigg|_{v=0}\right|dtdy\nonumber\\
&\hspace{-5cm}\quad +\int_{(0,\infty)\setminus B(x,1)}|f(y)|\int_0^\infty W_t^\lambda (x,y)\left|F(t,s)-\sum_{j=1}^m\frac{s^j}{j!}\partial_v^jF(t,v)\Bigg|_{v=0}\right|dtdy\nonumber\\
&\hspace{-5cm}\quad +|f(x)|\left(\int_{(0,\infty)\setminus B(x,1)}\int_0^1 W_t^\lambda (x,y)\left|F(t,s)-\sum_{j=1}^m\frac{s^j}{j!}\partial_v^jF(t,v)\Bigg|_{v=0}\right|dtdy\right.\nonumber\\
&\hspace{-5cm}\quad +\int_{B(x,1)}\int_1^\infty W_t^\lambda (x,y)\left|F(t,s)-\sum_{j=1}^m\frac{s^j}{j!}\partial_v^jF(t,v)\Bigg|_{v=0}\right|dtdy\nonumber\\
&\hspace{-5cm}\quad +\left.\int_0^\infty\int_0^1|W_t^\lambda (x,y)-W_t(x-y)|\left|F(t,s)-\sum_{j=1}^m\frac{s^j}{j!}\partial_v^jF(t,v)\Bigg|_{v=0}\right|dtdy\right)\nonumber\\
&\hspace{-5cm} \leq Cs^{m+1}\Bigg[\int_{B(x,1)}|f(y)-f(x)|\Bigg(\int_0^1W_t^\lambda (x,y)(1+|\log t|^{m+1})\frac{dt}{t^{1+s}}\nonumber\\
&\hspace{-5cm}\quad +\int_1^\infty W_t^\lambda (x,y)(1+|\log t|^{m+1})\frac{dt}{t}\Bigg)dy\nonumber\\
&\hspace{-5cm} \quad +\int_{(0,\infty)\setminus B(x,1)}|f(y)|\Bigg(\int_0^1W_t^\lambda (x,y)(1+|\log t|^{m+1})\frac{dt}{t^{1+s}}\nonumber\\
&\hspace{-5cm}\quad +\int_1^\infty W_t^\lambda (x,y)(1+|\log t|^{m+1})\frac{dt}{t}\Bigg)dy\nonumber\\
&\hspace{-5cm}\quad +|f(x)|\left(1+\int_{(0,\infty)\setminus B(x,1)}\int_0^1 W_t^\lambda (x,y)(1+|\log t|^{m+1})\frac{dt}{t^{1+s}}dy\right.\nonumber\\
&\hspace{-5cm}\quad +\int_{B(x,1)}\int_1^\infty W_t^\lambda (x,y)(1+|\log t|^{m+1})\frac{dt}{t}dy\nonumber\\
&\hspace{-5cm}\quad +\left.\int_0^\infty\int_0^1|W_t^\lambda (x,y)-W_t(x-y)|(1+|\log t|^{m+1})\frac{dt}{t^{1+s}}dy\right)\Bigg].
\end{align}
By proceeding as in the proof of Theorems~\ref{thm: Th1.1} and \ref{thm: Th1.2} (see Remark~\ref{rem: log}) we can get that, for each $x\in (0,\infty)$ and $s\in (0,\tfrac{\theta}{4})$,
\begin{align*}
    \left|B_\lambda^sf(x)-\left(f(x)+\sum_{j=1}^m\frac{s^j}{j!}(\log^jB_\lambda)(f)(x)\right)\right|&\\
    &\hspace{-1cm}\leq C\left(\|f\|_{L^\infty (0,\infty)}+\|f\|_{L^1(0,\infty)}+\|f\|_{{\textup{Lip}}^\theta_{{\rm loc}, {\rm uni}}(0,\infty)}\right)s^{m+1}.
\end{align*}
Thus, we have proved that
\[
\lim_{s\rightarrow 0^+}\frac{(m+1)!}{s^{m+1}}\left(B_\lambda^sf-f-\sum_{j=1}^m\frac{s^j}{j!}(\log^jB_\lambda)(f)\right)=(\log^{m+1}B_\lambda)f,\quad \mbox{ in }L^\infty (0,\infty).
\]

We now consider $1<p<\infty$. According to the estimates in Section~\ref{sec: proofs Th1.1 Th1.2}, \eqref{eq: proof Thm1.1 first term}, we have that for each $x\in (0,\infty)$ and $s\in (0,\tfrac{\theta}{4})$
\begin{align}\label{4.0}
    \int_{B(x,1)}|f(x)-f(y)|\int_0^\infty W_t^\lambda (x,y)\left|F(t,s)-\sum_{j=1}^m\frac{s^j}{j!}\partial_v^jF(t,v)\Bigg|_{v=0}\right|dtdy&\leq Cs^{m+1}\|f\|_{{\textup{Lip}}^\theta _{{\rm loc},{\rm uni}}(0,\infty)}.
\end{align}
Also, the estimates obtained in \eqref{eq: proof Thm1.1 second term} in Section~\ref{sec: proofs Th1.1 Th1.2} lead to
\begin{align}\label{F}
    \int_{(0,\infty)\setminus B(x,1)}|f(y)|\int_0^\infty W_t^\lambda (x,y)\left|F(t,s)-\sum_{j=1}^m\frac{s^j}{j!}\partial_v^jF(t,v)\Bigg|_{v=0}\right|dtdy&\leq Cs^{m+1}\|f\|_{L^1(0,\infty)}.
\end{align}
for every $x\in (0,\infty)$ and $s\in (0,\tfrac12)$.

On the other hand, taking $R>1$ such that $\supp f\subset B(0,R)$ and using again \eqref{eq: proof Thm1.1 first term} and \eqref{eq: proof Thm1.1 second term} in Section~\ref{sec: proofs Th1.1 Th1.2} we get
\begin{align*}
    \int_{B(x,1)}|f(x)-f(y)|&\int_0^\infty W_t^\lambda (x,y)\left|F(t,s)-\sum_{j=1}^m\frac{s^j}{j!}\partial_v^jF(t,v)\Bigg|_{v=0}\right|dtdy\\
     &\quad +\int_{(0,\infty)\setminus B(x,1)}|f(y)|\int_0^\infty W_t^\lambda (x,y)\left|F(t,s)-\sum_{j=1}^m\frac{s^j}{j!}\partial_v^jF(t,v)\Bigg|_{v=0}\right|dtdy\\
     &\leq \int_{B(0,R)}|f(y)|\int_0^\infty W_t^\lambda (x,y)\left|F(t,s)-\sum_{j=1}^m\frac{s^j}{j!}\partial_v^jF(t,v)\Bigg|_{v=0}\right|dtdy\\
    &\leq Cs^{m+1}\|f\|_{L^\infty (0,\infty)}\int_{B(0,R)}\left(\frac{1}{|x-y|^{1-2\varepsilon}}+\frac{1}{|x-y|^{2+2\varepsilon }}\right)dy\\
    &\leq Cs^{m+1}\|f\|_{L^\infty (0,\infty)}\frac{R}{x^{1-2\varepsilon}},\quad x>2R\mbox{ and }s\in (0,\tfrac12),
\end{align*}
provided that $\varepsilon \in (0,\tfrac12)$.
Then
\begin{align}\label{G}
&\left\|\int_{B(x,1)}|f(x)-f(y)|\int_0^\infty W_t^\lambda (x,y)\left|F(t,s)-\sum_{j=1}^m\frac{s^j}{j!}\partial_v^jF(t,v)\Bigg|_{v=0}\right|dtdy\right.&\nonumber\\
     &\hspace{1.2cm}\quad \left.+\int_{(0,\infty)\setminus B(x,1)}|f(y)|\int_0^\infty W_t^\lambda (x,y)\left|F(t,s)-\sum_{j=1}^m\frac{s^j}{j!}\partial_v^jF(t,v)\Big|_{v=0}\right|dtdy\right\|_{L^p((0,\infty)\setminus B(0,2R))}\nonumber\\
      &\hspace{1.2cm}\leq Cs^{m+1}\|f\|_{L^\infty (0,\infty)}R^{1/p-2\varepsilon},
\end{align}
whenever $0<\varepsilon<\min\left\{\tfrac{p-1}{2p},\tfrac{1}{2}\right\}$ and $s\in (0,\tfrac12)$.

Combining \eqref{4.0}, \eqref{F}  and \eqref{G} we conclude that, for each $s\in (0,\tfrac{\theta}{4})$,
\begin{align*}
\left\|\int_{B(x,1)}|f(x)-f(y)|\int_0^\infty W_t^\lambda (x,y)\left|F(t,s)-\sum_{j=1}^m\frac{s^j}{j!}\partial_v^jF(t,v)\Bigg|_{v=0}\right|dtdy\right.&\\
     &\hspace{-10cm}\quad \left.+\int_{(0,\infty)\setminus B(x,1)}|f(y)|\int_0^\infty W_t^\lambda (x,y)\left|F(t,s)-\sum_{j=1}^m\frac{s^j}{j!}\partial_v^jF(t,v)\Bigg|_{v=0}\right|dtdy\right\|_{L^p((0,\infty)}\leq Cs^{m+1}.
\end{align*}
From \eqref{eq: estimate limit Blambda s positive}, by proceeding as in the estimates \eqref{eq: proof Thm1.1 third term}-\eqref{eq: proof Thm1.1 last term}  in Section~\ref{sec: proofs Th1.1 Th1.2}, it follows that
\[
\left\|B_\lambda^sf-\big(f+\sum_{j=1}^n\frac{s^j}{j!}(\log^jB_\lambda)f\big)\right\|_{L^p(0,\infty)}\leq Cs^{n+1},\quad s\in \left(0,\tfrac{\theta}{4}\right),
\]
and
\[
\lim_{s\rightarrow 0^+}\frac{(m+1)!}{s^{m+1}}\left(B_\lambda^sf-\left(f+\sum_{j=1}^m\frac{s^j}{j!}(\log^jB_\lambda)f\right)\right)=(\log^{m+1}B_\lambda)f,\quad \mbox{ in }L^p(0,\infty).
\]
Finally, taking into account \eqref{eq: decomposition Bessel negative power} and \eqref{eq: representation log^m negative s} and proceeding in a similar way we can see that if $1<p\leq \infty$, then
\[
\left\|B_\lambda^{-s}f-\left(f+\sum_{j=1}^m\frac{(-1)^js^j}{j!}(\log^jB_\lambda)f\right)\right\|_{L^p(0,\infty)}\leq Cs^{m+1},\quad s\in \big(0,\tfrac{\theta}{4}\big),
\]
and
\[
\lim_{s\rightarrow 0^+}\frac{(m+1)!}{s^{m+1}}\left(B_\lambda^{-s}f-\left(f+\sum_{j=1}^m\frac{(-1)^js^j}{j!}(\log^jB_\lambda)f\right)\right)=(\log^{m+1}B_\lambda)f,\quad \mbox{ in }L^p(0,\infty).
\]


\section{Proof of Theorem~\ref{thm: Th1.4}}\label{sec: proof Th1.4}

Suppose that $f\in L^1\big((0,\infty),\frac{dy}{1+y}\big)$ and recall the definition of the extension of $f$ to $(0,\infty)\times (0,\infty)$ given by
\[
u_f^\lambda (x,t)=\int_0^\infty W_t^\lambda (f)(x)e^{-\frac{t^2}{4u}}\frac{du}{u},\quad x,t\in (0,\infty).
\]
We note first that the integral above is finite from \eqref{eq: heat Bessel FeynmanKac}:
\begin{align}
\int_0^\infty \int_0^\infty W_u^\lambda (x,y)|f(y)|e^{-\frac{t^2}{4u}}\frac{du}{u}dy&\leq C\int_0^\infty \int_0^\infty \frac{e^{-\frac{t^2+|x-y|^2}{4u}}}{u^{3/2}}du|f(y)|dy\nonumber\\
&\hspace{-4cm}\leq C\int_0^\infty \frac{|f(y)|}{(t^2+|x-y|^2)^{1/2}}dy \label{eq: uflambda poisson}\\
&\hspace{-4cm}= C\left(\int_0^{2x}+\int_{2x}^\infty\right)\frac{|f(y)|}{(t^2+|x-y|^2)^{1/2}}dy\nonumber\\
&\hspace{-4cm}\leq C\left(\frac{1+x}{t}\int_0^\infty \frac{|f(y)|}{1+y}dy+\int_0^\infty \frac{|f(y)|}{t+y}dy\right)\nonumber\\
&\hspace{-4cm}\leq C\left(\frac{1+x}{t}+\frac{1}{\min\{1,t\}}\right)\int_0^\infty \frac{|f(y)|}
{1+y}dy<\infty,\quad x,t\in (0,\infty). \label{eq: first estimate uflambda}
\end{align}
Therefore, we can interchange the order of integration to write
\[
u_f^\lambda (x,t)=\int_0^\infty f(y)\int_0^\infty W_u^\lambda (x,y)\frac{e^{-\frac{t^2}{4u}}}{u}dudy,\quad x,t\in (0,\infty).
\]

We first establish the limit at infinity in part \ref{itm: Th1.4 b}. From \eqref{eq: first estimate uflambda} we deduce that
\begin{align*}
\int_0^\infty |f(y)|\int_0^\infty W_u^\lambda (x,y)e^{-\frac{t^2}{4u}}\frac{du}{u}dy&\leq C\int_0^\infty |f(y)|\int_0^\infty \frac{e^{-\frac{1+(x-y)^2}{4u}}}{u^{3/2}}dudy\\
&\hspace{-4cm}\leq C\int_0^\infty \frac{|f(y)|}{(1+(x-y)^2)^{1/2}}dy\leq C(1+x)\int_0^\infty \frac{|f(y)|}{1+y}dy<\infty ,\quad x\in (0,\infty)\mbox{ and }t\geq 1.
\end{align*}
Using the dominated convergence theorem we get $\lim_{t\rightarrow +\infty}u_f^\lambda (x,t)=0$, for every $x\in (0,\infty)$. 

We now return to the proof of part \ref{itm: Th1.4 a}. 
By proceeding as in the proof of \cite[Lemma 3.2]{ChHW} we can see that $u_f^\lambda (\cdot, t)\in L^1\big((0,\infty),\frac{dy}{(1+y)^{1+\sigma}}\big)$, and 
\[
\|u_f^\lambda (\cdot ,t)\|_{L^1_\sigma(0,\infty)}\leq C(1+\log^{-}t),\quad t,\sigma \in (0,\infty).
\]
where we recall that $\log^-t=\max\{0,-\log t\}$, $t>0$. 

Moreover, since $f\in L^1\big((0,\infty),\frac{dy}{1+y}\big)$, by \eqref{eq: uflambda poisson}
\[
|u_f^\lambda (x,t)|\leq C\int_0^\infty \frac{|f(y)|}{(t^2+|x-y|^2)^{1/2}}dt,\quad x,t\in (0,\infty),
\]
and applying \cite[Theorem 3.1, $(i)\Rightarrow (ii)$]{ChHW} we deduce that $u_f^\lambda\in AG((0,\infty)\times (0,\infty))$.

To see that $u_f^\lambda\in C^2((0,\infty)\times (0,\infty))$, we first justify differentiation under the integral sign with respect to the variable $t$. More precisely, for $j=1,2$,
\begin{equation}\label{eq: time derivatives uflambda}
\partial_t^ju_f^\lambda (x,t)=\int_0^\infty f(y)\int_0^\infty W_u^\lambda (x,y)\frac{\partial_t^je^{-\frac{t^2}{4u}}}{u}dudy,\quad x,t\in (0,\infty)
\end{equation}

Let $x_0,t_0\in (0,\infty)$. According to \eqref{eq: heat Bessel FeynmanKac} we get
\begin{align*}
    \int_0^\infty |f(y)|\int_0^\infty W_u^\lambda (x_0,y)\frac{|\partial_te^{-\frac{t^2}{4u}}|}{u}dudy&\leq C \int_0^\infty |f(y)|\int_0^\infty \frac{e^{-\frac{|x_0-y|^2}{4u}}}{u^{5/2}}te^{-\frac{t^2}{4u}}dudy\\
    &\hspace{-4cm}\leq Ct\int_0^\infty \frac{|f(y)|}{(|x_0-y|^2+t^2)^{3/2}}dy\leq C\int_0^\infty \frac{|f(y)|}{|x_0-y|^2+t^2}dy\\
    &\hspace{-4cm}\leq C\left(\int_0^{2x_0}+\int_{2x_0}^\infty\right)\frac{|f(y)|}{|x_0-y|^2+t^2}dy\\
    &\hspace{-4cm}\leq C\left(\frac{(1+x_0)^2}{t^2}+\frac{1}{\min\{1,t^2\}}\right)\int_0^\infty \frac{|f(y)|}{(1+y)^2}dy\\
    &\hspace{-4cm}\leq C\int_0^\infty \frac{|f(y)|}{1+y}dy,\quad t\in \left(\tfrac{t_0}{2},3\tfrac{t_0}{2}\right),
\end{align*}
and
\begin{align*}
    \int_0^\infty |f(y)|\int_0^\infty W_u^\lambda (x_0, y)\frac{|\partial_t^2e^{-\frac{t^2}{4u}}|}{u}dudy&\leq C \int_0^\infty |f(y)|\int_0^\infty \frac{e^{-\frac{|x_0-y|^2}{4u}}}{u^{3/2}}\left(\frac{1}{u}+\frac{t^2}{u^2}\right)e^{-\frac{t^2}{4u}}dudy\\
    &\hspace{-4cm}\leq C\int_0^\infty |f(y)|\int_0^\infty \frac{e^{-c\frac{(x_0-y)^2+t^2}{u}}}{u^{5/2}}dudy\\
    &\hspace{-4cm}\leq \frac{C}{t}\left(\frac{(1+x_0)^2}{t^2}+\frac{1}{\min\{1,t^2\}}\right)\int_0^\infty \frac{|f(y)|}{(1+y)^2}dy\\
    &\hspace{-4cm}\leq C\int_0^\infty \frac{|f(y)|}{1+y}dy,\quad t\in \big(\tfrac{t_0}{2},3\tfrac{t_0}{2}\big),
\end{align*}
Thus, the equality in \eqref{eq: time derivatives uflambda} is proved.

Observe that
\[
\left(\partial_t^2+\frac{1}{t}\partial_t\right)e^{-\frac{t^2}{4u}}=\left(-\frac{1}{u}+\frac{t^2}{4u^2}\right)e^{-\frac{t^2}{4u}}=u\partial_u\left(\frac{e^{-\frac{t^2}{4u}}}{u}\right),\quad t,u\in (0,\infty).
\]
Thus, using \eqref{eq: time derivatives uflambda} we deduce that
\[
\left(\partial_t^2+\frac{1}{t}\partial_t\right)u_f^\lambda (x,t)=\int_0^\infty f(y)\int_0^\infty W_u^\lambda(x,y)\partial_u\left(\frac{e^{-\frac{t^2}{4u}}}{u}\right)dudy,\quad x,t\in (0,\infty). 
\]
Since $\partial_uW_u^\lambda (x,y)=B_{\lambda ,x}W_u^\lambda (x,y)$, $x,y,u\in (0,\infty)$, partial integration leads to
\[
\int_0^\infty W_u^\lambda (x,y)\partial_u\left(\frac{e^{-\frac{t^2}{4u}}}{u}\right)du=W_u^\lambda (x,y)\frac{e^{-\frac{t^2}{4u}}}{u}\Bigg|_{u=0}^{u=\infty}-\int_0^\infty B_{\lambda ,x}W_u^\lambda (x,y)\frac{e^{-\frac{t^2}{4u}}}{u}du,\quad x,y,t\in (0,\infty).
\]
According to \eqref{eq: heat Bessel FeynmanKac} it follows that $W_u^\lambda (x,y)\frac{e^{-\frac{t^2}{4u}}}{u}\Bigg|_{u=0}^{u=\infty}=0$, $x,y,t\in (0,\infty)$. Hence,
\[
\left(\partial_t^2+\frac{1}{t}\partial_t\right)u_f^\lambda (x,t)=\int_0^\infty f(y)\int_0^\infty B_{\lambda ,x}W_u^\lambda (x,y)\frac{e^{-\frac{t^2}{4u}}}{u}dudy,\quad x,t\in (0,\infty).
\]

We will now show that, for $x,t\in (0,\infty)$,
\begin{equation}\label{eq: Blambdax exchange}
   \int_0^\infty f(y)\int_0^\infty B_{\lambda ,x}W_u^\lambda (x,y)\frac{e^{-\frac{t^2}{4u}}}{u}dudy=B_{\lambda ,x}\int_0^\infty f(y)\int_0^\infty W_u^\lambda (x,y)\frac{e^{-\frac{t^2}{4u}}}{u}dudy.
\end{equation}

Note we can write $B_{\lambda, x}=x^{-\lambda -1/2}\frac{d}{dx}x^{2\lambda +1}\frac{d}{dx}x^{-\lambda -1/2}$. 
By using \eqref{eq: Ilambda derivative}, for $x,t\in (0,\infty)$,
\begin{align*}
    \partial_x\left(x^{-\lambda -1/2}W_t^\lambda (x,y)\right)&=\partial_x \left(y^{\lambda+1/2}\frac{(xy)^{-\lambda }}{2t}I_\lambda \left(\frac{xy}{2t}\right)e^{-\frac{x^2+y^2}{4t}}\right)\\
    &=\left(\frac{y^{\lambda +3/2}}{(2t)^{\lambda +1}}\left(\frac{xy}{2t}\right)^{-\lambda }I_{\lambda +1}\left(\frac{xy}{2t}\right)-\frac{y^{\lambda+1/2}}{(2t)^\lambda }\left(\frac{xy}{2t}\right)^{-\lambda }I_\lambda \left(\frac{xy}{2t}\right)\frac{x}{2t}\right)\frac{e^{-\frac{x^2+y^2}{4t}}}{2t}.
\end{align*}
Taking into account the behaviour of $I_\lambda (z)$ when $z\rightarrow 0^+$ and $z\rightarrow +\infty$ (see \eqref{eq: Ilambda at zero} and \eqref{eq: Ilambda at infty}) we deduce that
\begin{align*}
    \left|\partial_x\left(x^{-\lambda -1/2}W_t^\lambda (x,y)\right)\right|&\leq C\left(\frac{y^{\lambda +3/2}}{t^{\lambda +1}}\frac{xy}{t}+\frac{y^{\lambda+1/2}}{t^\lambda }\frac{x}{2t}\right)\frac{e^{-c\frac{x^2+y^2}{t}}}{t}\leq C\frac{xy^{\lambda+1/2}}{t^{\lambda +2}}\left(\frac{y^2}{t}+1\right)e^{-c\frac{x^2+y^2}{t}}\\
    &\leq C\frac{x}{t^{\lambda /2+7/4}}e^{-c\frac{x^2+y^2}{t}},\quad x,y,t\in (0,\infty) \mbox{ and }xy<t,
\end{align*}
and
\begin{align*}
    \left|\partial_x\left(x^{-\lambda -1/2}W_t^\lambda (x,y)\right)\right|&=\frac{y^{\lambda+1/2}}{(2t)^{\lambda +2}}\left(\frac{xy}{2t}\right)^{-\lambda }\left|I_{\lambda +1}\left(\frac{xy}{2t}\right)y-I_\lambda \left(\frac{xy}{2t}\right)x\right|e^{-\frac{x^2+y^2}{4t}}\\
    &\leq C\frac{y^{\lambda+1/2}}{t^{\lambda +2}}\left(\frac{xy}{t}\right)^{-\lambda -\tfrac12}(y+x)e^{-c\frac{(x-y)^2}{t}},\quad x,y,t\in (0,\infty) \mbox{ and }xy>t.
\end{align*}

Let $x_0,t_0\in (0,\infty)$. We have that
\begin{align*}
    \int_0^ \infty |f(y)|\int_0^\infty \left|\partial_x\left(x^{-\lambda -1/2}W_u^\lambda (x,y)\right)\right|\frac{e^{-\frac{t^2}{4u}}}{u}dudy&\leq C\left(x\int_0^\infty |f(y)|\int_{xy}^\infty \frac{e^{-c\frac{x^2+y^2}{u}}}{u^{\lambda/2+11/4}}dudy\right.\\
    &\hspace{-6cm}\left.\quad +\int_0^\infty |f(y)|(x+y)y^{\lambda+1/2}\int_0^{xy}\left(\frac{xy}{u}\right)^{-\lambda -\tfrac12}\frac{e^{-c\frac{(x-y)^2+t_0^2}{u}}}{u^{\lambda +3}}dudy\right)\\
    &\hspace{-6cm}\leq C\left(x_0\int_0^\infty \frac{|f(y)|}{(x_0^2+y^2)^{\lambda/2+7/4}}dy+\int_0^\infty|f(y)|\int_0^{xy}(y+x_0)\frac{y^{\lambda+1/2}}{u^{\lambda+3}}e^{-c\frac{(x-y)^2+t^2_0}{u}}dudy\right.\\
    &\hspace{-6cm}\le C\Big(\frac{1}{x_0^{\lambda+3/2}}\int_0^\infty \frac{|f(y)|}{x_0+y}dy+\int_0^\infty|f(y)|(y+x_0)\frac{y^{\lambda+1/2}}{((x-y)^2+t_0^2)^{\lambda+2}}dy\Big)\\
    &\hspace{-6cm}\le C\left(\frac{1}{x_0^{\lambda+3/2}}\int_0^\infty \frac{|f(y)|}{x_0+y}dy+\left(\int_0^{2x}+\int_{2x}^\infty\right)|f(y)|(y+x_0)\frac{y^{\lambda+1/2}}{((x-y)^2+t_0^2)^{\lambda+2}}dy\right)\\
    &\hspace{-6cm}\le C\Bigg(\frac{1}{x_0^{\lambda+3/2}}\int_0^\infty \frac{|f(y)|}{x_0+y}dy+\frac{x_0^{\lambda+3/2}(1+x_0)}{t_0^{2\lambda+4}}\int_0^\infty\frac{|f(y)|}{1+y}dy \\
    &\hspace{-6cm}\quad +\int_{2x}^\infty|f(y)|\frac{y^{\lambda+3/2}}{(y^2+t_0^2)^{\lambda+2}}dy\Bigg)\\
    &\hspace{-6cm}\le C\left(\frac{1}{x_0^{\lambda+3/2}}+\frac{(1+x_0)x_0^{\lambda+3/2}}{t_0^{2\lambda+4}}\right)\left(\int_0^\infty\frac{|f(y)|}{1+y}dy+\frac{1}{t_0^\lambda}\int_{2x}^\infty\frac{|f(y)|}{(y+t_0)^{5/2}}dy\right)\\
    &\hspace{-6cm}\le C\int_0^\infty\frac{|f(y)|}{1+y}dy,\quad x\in \big(\tfrac{x_0}{2},3\tfrac{x_0}{2}\big).
\end{align*}
Here $C>0$ depends only on $x_0$ and $t_0$.

The last estimate leads to
\begin{align*}
\partial_x\left(x^{-\lambda -1/2}\int_0^\infty f(y)\int_0^\infty W_u^\lambda (x,y)\frac{e^{-\frac{t^2}{4u}}}{u}dudy\right)&=\int_0^\infty f(y)\int_0^\infty \partial_x\left(x^{-\lambda -1/2}W_u^\lambda (x,y)\right)\frac{e^{-\frac{t^2}{4u}}}{u}dudy,
\end{align*}
for $x,t\in (0,\infty)$.

On the other hand, by \eqref{eq: Ilambda derivative}
\begin{align*}
x^{-\lambda -1/2}\partial_x\left(x^{2\lambda +1}\partial_x\left(x^{-\lambda -1/2}W_u(x,y)\right)\right)&=B_{\lambda ,x}W_u^\lambda (x,y)=\partial_uW_u^\lambda (x,y)\\
&\hspace{-4cm}=\partial_u\left(\frac{\sqrt{xy}}{2u}I_\lambda \left(\frac{xy}{2u}\right)e^{-\frac{x^2+y^2}{2u}}\right)= \partial_u\left(\frac{(xy)^{\lambda+1/2}}{(2u)^{\lambda +1}}\left(\frac{xy}{2u}\right)^{-\lambda }I_\lambda \left(\frac{xy}{2u}\right)e^{-\frac{x^2+y^2}{2u}}\right)\\
&\hspace{-4cm}=(xy)^{\lambda+1/2}\Bigg(-\frac{2(\lambda+1)}{(2u)^{\lambda +2}}\left(\frac{xy}{2u}\right)^{-\lambda }I_\lambda \left(\frac{xy}{2u}\right)-\frac{1}{(2u)^{\lambda +1}}\frac{xy}{2u^2}\left(\frac{xy}{2u}\right)^{-\lambda }I_{\lambda+1} \left(\frac{xy}{2u}\right)\\
&\hspace{-4cm}\quad +\frac{x^2+y^2}{4u^2}\frac{1}{(2u)^{\lambda +1}}\left(\frac{xy}{2u}\right)^{-\lambda }I_\lambda \left(\frac{xy}{2u}\right)\Bigg)e^{-\frac{x^2+y^2}{2u}},\quad x,u\in (0,\infty).
\end{align*}
Using again \eqref{eq: Ilambda at zero}, for $x,y,u\in (0,\infty)$ and $xy\leq u$,
\begin{align*}
|B_{\lambda ,x}W_u^\lambda (x,y)|&\leq C\frac{(xy)^{\lambda+1/2}}{u^{\lambda +2}}\left(1+\left(\frac{xy}{u}\right)^2+\frac{x^2+y^2}{u}\right)e^{-c\frac{x^2+y^2}{u}}\leq C\frac{(xy)^{\lambda+1/2}}{u^{\lambda +2}}e^{-c\frac{x^2+y^2}{u}},  
\end{align*}
and by \eqref{eq: Ilambda at infty}, when $xy>u$,
\begin{align*}
|B_{\lambda ,x}W_u^\lambda (x,y)|&\leq C\Bigg[\frac{1}{u^{3/2}}\sqrt{\frac{xy}{2u}}I_\lambda \left(\frac{xy}{2u}\right)+\frac{xy}{u^{5/2}}\sqrt{\frac{xy}{2u}}I_{\lambda +1} \left(\frac{xy}{2u}\right)\\
&\quad +\frac{x^2+y^2}{u^{5/2}}\sqrt{\frac{xy}{2u}}I_\lambda \left(\frac{xy}{2u}\right)\Bigg]e^{-c\frac{(x-y)^2}{u}}\\
&\leq C\left(\frac{1}{u^{3/2}}+\frac{xy}{u^{5/2}}+\frac{x^2+y^2}{u^{5/2}}\right)e^{-\frac{(x-y)^2}{4u}}. 
\end{align*}
Let $x_0,t_0\in (0,\infty)$. It follows that
\begin{align*}
    \int_0^ \infty |f(y)|\int_0^\infty |\partial_u W_u^\lambda (x,y)|\frac{e^{-\frac{t_0^2}{4u}}}{u}dudy&\leq C\Bigg(\int_0^\infty |f(y)|\int_{xy}^\infty (xy)^{\lambda+1/2}\frac{e^{-c\frac{x^2+y^2+t_0^2}{u}}}{u^{\lambda +3}}dudy\\
    &\hspace{-6cm}\quad +\int_0^\infty |f(y)|\int_0^{xy}\left(\frac{1}{u^{3/2}}+\frac{x^2+y^2}{u^{5/2}}\right)e^{-c\frac{(x-y)^2+t_0^2}{u}}\frac{du}{u}dy\Bigg)\\
    &\hspace{-6cm}\leq C\left[\int_0^\infty |f(y)|\frac{(x_0y)^{\lambda+1/2}}{(x_0^2+y^2+t_0^2)^{\lambda +3}}dy\right.\\
    &\hspace{-6cm}\left.\quad +\left(\int_0^{2x}+\int_{2x}^\infty\right)|f(y)|\int_0^{xy}\left(\frac{1}{u^{5/2}}+\frac{x^2+y^2}{u^{7/2}}\right)e^{-c\frac{(x-y)^2+t_0^2}{u}}dudy\right]\\
    &\hspace{-6cm}\leq C\left[\int_0^\infty \frac{|f(y)|}{1+y}dy+\int_{x_0}^\infty|f(y)|\left(\frac{1}{(y^2+t_0^2)^{3/2}}+\frac{x_0^2+y^2}{(y^2+t_0^2)^{5/2}}\right)dy\right]\\
    &\hspace{-6cm}\leq C\int_0^\infty \frac{|f(y)|}{1+y}dy,\quad x\in \big(\tfrac{x_0}{2},3\tfrac{x_0}{2}\big),
\end{align*}
where $C>0$ depends only on $x_0$ and $t_0$.

The above estimates yields \eqref{eq: Blambdax exchange}. Thus, we have established that
\[
\left(\partial_t^2+\frac{1}{t}\partial_t\right)u_f^\lambda (x,t)=B_{\lambda ,x}u_f^\lambda (x,t),\quad x,t\in (0,\infty).
\]

The estimates above are uniform for $(x,t)\in (\tfrac{x_0}{2},3\tfrac{x_0}{2})\times(\tfrac{t_0}{2},3\tfrac{t_0}{2})$ of any point $(x_0,t_0)\in(0,\infty)\times(0,\infty)$.
Consequently, by the theorem on differentiation under the integral sign, the functions
$(x,t)\mapsto \partial_t^ju_f^\lambda (x,t)$, $j=1,2$, and $(x,t)\mapsto B_{\lambda ,x}u_f^\lambda (x,t)$
are continuous on $(0,\infty)\times (0,\infty)$. This, together with the continuity of
$u_f^\lambda$ itself (which follows from \eqref{eq: first estimate uflambda} in the same way), shows that
$u_f^\lambda\in C^2((0,\infty)\times (0,\infty))$.

We are now going to see that
\[
\lim_{t\rightarrow 0^+}t\partial_t u_f^\lambda (x,t)=-2f(x),\quad \mbox{ in  }L^1_{\rm loc}(0,\infty).
\]

Note that
\[
t\partial_t u_f^\lambda (x,t)=-\int_0^\infty f(y)\int_0^\infty W_u^\lambda (x,y)\frac{t^2}{2u}\frac{e^{-\frac{t^2}{4u}}}{u}dudy,\quad x,t\in (0,\infty).
\]
We decompose $t\partial_t u_f^\lambda (x,t)$ as follows
\begin{align}\label{P1}
    -t\partial_t u_f^\lambda (x,t)&=\left(\int_0^{x/2}+\int_{2x}^\infty \right)f(y)\int_0^\infty W_u^\lambda (x,y)\frac{t^2}{2u^2}e^{-\frac{t^2}{4u}}dudy\nonumber\\
    &\quad +\int_{x/2}^{2x}f(y)\int_0^ \infty (W_u^\lambda (x,y)-W_u(x-y))\frac{t^2}{2u^2}e^{-\frac{t^2}{4u}}dudy\nonumber\\
    &\quad +\int_{x/2}^{2x}f(y)\int_0^ \infty W_u(x-y)\frac{t^2}{2u^2}e^{-\frac{t^2}{4u}}dudy\nonumber\\
    &:=\sum_{j=1}^4I_j(x,t),\quad x,t\in (0,\infty).
\end{align}
According to \eqref{eq: heat Bessel FeynmanKac} we obtain
\begin{align}\label{P2}
|I_1(x,t)|&\leq C\int_0^{x/2}|f(y)|\int_0^\infty \frac{e^{-\frac{|x-y|^2+t^2}{4u}}}{\sqrt{u}}\frac{t^2}{u^2}du\leq Ct^2\int_0^{x/2}\frac{|f(y)|}{((x-y)^2+t^2)^{3/2}}dy\nonumber\\
&\leq Ct^2\frac{1+x}{x^3}\int_0^\infty \frac{|f(y)|}{1+y}dy,\quad x,t\in (0,\infty),
\end{align}
and
\begin{align}\label{P3}
|I_2(x,t)|&\leq Ct^2\int_{2x}^\infty \frac{|f(y)|}{((x-y)^2+t^2)^{3/2}}dy\leq Ct^2\int_{2x}^\infty \frac{|f(y)|}{(x^2+y^2)^{3/2}}dy\nonumber\\
&\leq C\frac{t^2}{x^2}\left(1+\frac{1}{x}\right)\int_0^\infty \frac{|f(y)|}{1+y}dy,\quad  x,t\in (0,\infty).
\end{align}

Using \eqref{eq: heat Bessel FeynmanKac} and \eqref{eq: heat Bessel comparison t<xy} we get
\begin{align}\label{P4}
|I_3(x,t)|&\leq Ct^2\int_{x/2}^{2x}|f(y)|\left(\int_0^{xy}+\int_{xy}^\infty \right) |W_u^\lambda (x,y)-W_u(x-y)|\frac{e^{-\frac{t^2}{4u}}}{u^2}dudy\nonumber\\
&\leq Ct^2\int_{x/2}^{2x}|f(y)|\left(\int_0^{xy}\frac{e^{-c\frac{(x-y)^2}{u}}}{u^{3/2}}\frac{1}{xy}e^{-\frac{t^2}{4u}}dudy+\int_{xy}^\infty \frac{e^{-\frac{(x-y)^2+t^2}{4u}}}{u^{5/2}}dudy\right)\nonumber\\
&\leq Ct^2\int_{x/2}^{2x}|f(y)|\left(\frac{1}{xy}\frac{1}{|x-y|+t}+\frac{1}{(xy)^{3/2}}\right)\nonumber\\
&\leq Ct\left(\frac{1+x}{x^2}+t\frac{1+x}{x^3}\right)\int_0^\infty\frac{|f(y)|}{1+y}dy,\quad x,t\in (0,\infty).
\end{align}
On the other hand, proceeding as in the estimates of $I_1$ and $I_2$, we obtain
\begin{equation}\label{P5}
\left|\left(\int_0^{x/2}+\int_{2x}^\infty \right)f(y)\int_0^\infty W_u(x-y)\frac{t^2}{2u^2}e^{-\frac{t^2}{4u}}dudy\right|\leq C\frac{t^2}{x^2}\left(1+\frac{1}{x}\right)\int_0^\infty \frac{|f(y)|}{1+y}dy,
\end{equation}
for $x,t\in (0,\infty)$.

We consider the function $f_0$ defined by $f_0(x)=f(x)$, $x\in (0,\infty)$, and $f_0(x)=0$, $x\in (-\infty ,0]$. For $x,t\in (0,\infty)$ we have that 
\[
\int_0^\infty f(y)\int_0^\infty W_u(x-y)\frac{e^{-\frac{t^2}{4u}}}{u}dudy=\int_0^\infty \frac{f(y)}{((x-y)^2+t^2)^{1/2}}dy=\int_{-\infty}^{+\infty }\frac{f_0(y)}{((x-y)^2+t^2)^{1/2}}dy.
\]

According to \cite[Theorem 3.1, (3.22)]{ChHW} it follows that
\[
\lim_{t\rightarrow 0^+}t\partial_t\int_{-\infty }^{+\infty }\frac{f_0(y)}{((x-y)^2+t^2)^{1/2}}=-2f_0(x),\quad \mbox{ in }L^1_{\rm loc}(\mathbb R).
\]
Then
\begin{equation}\label{P6}
\lim_{t\rightarrow 0^+}t\partial_t\int_0^\infty f(y)\int_0^\infty W_u(x-y)\frac{e^{-\frac{t^2}{4u}}}{u}dudy=-2 f(x)\quad \mbox{ in }L^1_{\rm loc}(0,\infty).
\end{equation}
From \eqref{P5} and \eqref{P6} we deduce that
\begin{equation}\label{P7}
\lim_{t\rightarrow 0^+}I_4(x,t)=2 f(x),\quad \mbox{ in }L^1_{\rm loc}(0,\infty).
\end{equation}
Since the right-hand sides of \eqref{P2}--\eqref{P4} are uniformly bounded on compact subsets of $(0,\infty)$ by a constant times a polynomial in $t$, it follows that $I_j(\cdot,t)\to 0$ in $L^1_{\rm loc}(0,\infty)$, $j=1,2,3$. Combining this with \eqref{P7} we conclude that
\[
-\lim_{t\rightarrow 0^+}t\partial_tu_f^\lambda (x,t)=2 f(x),\quad \mbox{ in }L^1_{\rm loc}(0,\infty).
\]
Since $u_f^\lambda\in AG((0,\infty)\times (0,\infty))\cap C^2((0,\infty)\times (0,\infty))$ and $f\in L^1_{\rm loc}(0,\infty)$, proceeding as in the proof of \cite[Lemma 3.6]{ChHW} we obtain
\[
\lim_{t\rightarrow 0^+}\frac{u_f^\lambda (\cdot,t)}{\log t}=-2 f(\cdot),\quad \mbox{ in }L^1_{\rm loc}(0,\infty).
\]

Suppose now that $f$ is Dini continuous in $x\in (0,\infty)$. We can write 
\begin{align}\label{L0}
u_f^\lambda (x,t)&=\int_{(0,\infty)\setminus B(x,1)}f(y)\int_0^\infty \frac{e^{-\frac{t^2}{4u}}-1}{u}W_u^\lambda (x,y)dudy\nonumber\\
&\quad +\int_{(0,\infty)\setminus B(x,1)}f(y)\int_0^\infty \frac{W_u^\lambda (x,y)}{u}dudy\nonumber\\
&\quad +\int_{B(x,1)\setminus B(x,t)}f(y)\int_0^\infty \frac{e^{-\frac{t^2}{4u}}}{u}W_u^\lambda (x,y)dudy\nonumber\\
&\quad +\int_{B(x,t)}f(y)\int_0^\infty \frac{e^{-\frac{t^2}{4u}}}{u}W_u^\lambda (x,y)dudy\nonumber\\
&:=I_1(t)+I_2+I_3(t)+I_4(t),\quad t\in (0,\infty).
\end{align}
Using \eqref{eq: heat Bessel FeynmanKac} we get
\begin{align*}
  |I_1(t)|&\leq C\int_{(0,\infty)\setminus B(x,1)}|f(y)|\int_0^\infty \frac{t^2}{u^{5/2}}e^{-\frac{(x-y)^2}{4u}}dudy\\
&\leq Ct^2\int_{(0,\infty)\setminus B(x,1)}\frac{|f(y)|}{|x-y|^3}dy= Ct^2\left(\int_{\substack{(0,\infty)\setminus B(x,1)\\y\leq 2x}}+\int_{\substack{(0,\infty)\setminus B(x,1)\\ y>2x}}\right)\frac{|f(y)|}{|x-y|^3}dy\\
&\leq Ct^2\left((1+x)\int_0^\infty \frac{|f(y)|}{1+y}dy+\int_{2x}^\infty \frac{|f(y)|}{y^3}dy\right)\\
&\leq Ct^2\left((1+x)\int_0^\infty \frac{|f(y)|}{1+y}dy+\int_{2x}^\infty \frac{|f(y)|}{(x+y)^3}dy\right)\\
&\leq Ct^2\left((1+x)\int_0^\infty \frac{|f(y)|}{1+y}dy+\frac{1}{\min\{1,x^3\}}\int_0^\infty \frac{|f(y)|}{(1+y)^3}dy\right)\\
&\leq Ct^2\left(1+x+\frac{1}{\min\{1,x^3\}}\right)\int_0^\infty \frac{|f(y)|}{1+y}dy,\quad t\in (0,\infty).
\end{align*}
Thus,
\begin{equation}\label{L1}
\lim_{t\rightarrow 0^+}I_1(t)=0.
\end{equation}

We now decompose $I_3$ as follows
\begin{align*}
    I_3(t)&=
    \int_{B(x,1)\setminus B(x,t)}(f(y)-f(x))\int_0^\infty \frac{W_u^\lambda (x,y)}{u}e^{-\frac{t^2}{4u}}dudy\\
    &\quad +f(x)\int_{B(x,1)\setminus B(x,t)}\int_0^\infty \frac{W_u^\lambda (x,y)}{u}e^{-\frac{t^2}{4u}}dudy\\
    &=I_{3,1}(t)+I_{3,2}(t),\quad t\in (0,1).
\end{align*}
Since $f$ is Dini continuous in $x\in (0,\infty)$, by \eqref{eq: heat Bessel FeynmanKac} we have that 
\begin{align*}
    |I_{3,1}(t)|&\le \int_{B(x,1)\setminus B(x,t)}|f(y)-f(x)|\int_0^\infty \frac{e^{-c\frac{|x-y|^2}{u}}}{u^{3/2}}dudy\\
    &\le C\int_{B(x,1)\setminus B(x,t)}\sup_{|x-z|\le|x-y|}\frac{|f(x)-f(z)|}{|x-z|}dy\\
    &\le C\int_0^1\frac{w_{f,x}(r)}{r}dr<\infty,\quad t\in (0,1).
\end{align*}
Dominated convergence theorem leads to
\[
\lim_{t\rightarrow 0^+}I_{3,1}(t)=\int_{B(x,1)}(f(y)-f(x))\int_0^\infty W_u^\lambda (x,y)\frac{du}{u}dy.
\]
We define 
\[
J_{3,2}(t)=\int_{B(x,1)\setminus B(x,t)}\int_0^\infty \frac{W_u^\lambda (x,y)}{u}e^{-\frac{t^2}{4u}}dudy,\quad t\in (0,1).
\]
We can write
\begin{align*}
J_{3,2}(t)&=\int_{B(x,1)\setminus B(x,t)}\int_1^\infty \frac{W_u^\lambda (x,y)}{u}e^{-\frac{t^2}{4u}}dudy\\
&\quad +\int_{B(x,1)\setminus B(x,t)}\int_0^1 \frac{W_u^\lambda (x,y)-W_u(x-y)}{u}e^{-\frac{t^2}{4u}}dudy \\
&\quad -\int_{B(x,1)\setminus B(x,t)}\int_1^\infty \frac{W_u(x-y)}{u}e^{-\frac{t^2}{4u}}dudy\\
&\quad +\int_{B(x,1)\setminus B(x,t)}\int_0^\infty \frac{W_u(x-y)}{u}e^{-\frac{t^2}{4u}}dudy\\
&=\sum_{j=1}^4J_{3,2,j}(t),\quad t\in (0,1).
\end{align*}
According to \eqref{eq: Ilambda at zero} and \eqref{eq: heat Bessel FeynmanKac} we get 
\begin{align*}
    |J_{3,2,1}(t)|&\leq C\int_{B(x,1)\setminus B(x,t)}\int_1^\infty \left(\frac{(xy)^{\lambda+1/2}}{u^{\lambda +2}}e^{-\frac{x^2+y^2}{4u}}+\frac{e^{-c\frac{(x-y)^2}{u}}}{u^{3/2}}\right)dudy\\
    &\leq C\left(\int_{B(x,1)}\frac{(xy)^{\lambda+1/2}}{(x^2+y^2)^{\lambda +1}}dy+\int_{B(x,1)}\int_1^\infty \frac{du}{u^{3/2}}dy\right)\\
    &\leq C\left(\int_{0}^\infty \frac{z^{\lambda+1/2}}{(1+z^2)^{\lambda +1}}dz+\int_{B(x,1)}\int_1^\infty \frac{du}{u^{3/2}}dy\right)\leq C,\quad t\in (0,1),
\end{align*}
and 
\[
|J_{3,2,3}(t)|\leq C\int_{B(x,1)\setminus B(x,t)}\int_1^\infty\frac{e^{-c\frac{(x-y)^2}{u}}}{u^{3/2}}du\leq C\int_{B(x,1)}dy\leq C,\quad t\in (0,1).
\]
The dominated convergence theorem leads to
\[
\lim_{t\rightarrow 0^+}(J_{3,2,1}(t)+J_{3,2,3}(t))=\int_{B(x,1)}\int_1^\infty \left(W_u^\lambda (x,y)-W_u(x-y)\right)\frac{du}{u}dy.
\]
For $J_{3,2,2}$, using \eqref{eq: heat Bessel FeynmanKac} and \eqref{eq: heat Bessel comparison t<xy}, we obtain the following estimates
\begin{align*}
    |J_{3,2,2}(t)|&\leq \int_{B(x,1)\setminus B(x,t)}\left(\int_0^{\min \{xy,1\}}|W_u^\lambda (x,y)-W_u(x-y)|\frac{e^{-\frac{t^2}{4u}}}{u}du\right.\\
    &\quad \left.+\mathcal X_{(0,1/x)}(y)\int_{xy}^1|W_u^\lambda (x,y)-W_u(x-y)|\frac{e^{-\frac{t^2}{4u}}}{u}du\right)dy\\
    &\leq C\int_{B(x,1)\setminus B(x,t)}\left(\int_0^{\min \{xy,1\}}\frac{e^{-c\frac{(x-y)^2+t^2}{u}}}{u^{1/2}}\frac{1}{xy}du\right.\\
    &\quad \left.+\mathcal X_{(0,1/x)}(y)\int_{xy}^1\frac{e^{-c\frac{(x-y)^2+t^2}{u}}}{u^{3/2}}du\right)dy\\
    &\leq C\int_{B(x,1)}\left(\frac{1}{xy}\int_0^{\min \{xy,1\}}\frac{du}{u^{1/2}}+\int_{xy}^\infty \frac{du}{u^{3/2}}\right)dy\\
    &\leq C\int_{B(x,1)}\frac{dy}{\sqrt{xy}}\leq C \sqrt{1+\frac{1}{x}},\quad t\in (0,1).
\end{align*}
Again, by using the dominated convergence theorem we get
\[
\lim_{t\rightarrow 0^+}J_{3,2,2}(t)=\int_{B(x,1)}\int_0^1\frac{W_u^\lambda (x,y)-W_u(x-y)}{u}dudy.
\]
On the other hand, we can write
\begin{align*}
    J_{3,2,4}(t)&=\int_{B(x,1)\setminus B(x,t)}\int_0^\infty \frac{W_u(x-y)}{u}e^{-\frac{t^2}{4u}}dudy=\int_{B(x,1)\setminus B(x,t)}\int_0^\infty \frac{e^{-\frac{(x-y)^2+t^2}{4u}}}{2\sqrt{\pi}u^{3/2}}dudy\\
    &=\int_{B(x,1)\setminus B(x,t)}\frac{dy}{((x-y)^2+t^2)^{1/2}}=\int_{(x-1,x+1)\setminus (x-t,x+t)}\frac{dy}{((x-y)^2+t^2)^{1/2}}\\
    &\quad -\mathcal X_{(0,1)}(x)\int_{x-1}^0\frac{dy}{((x-y)^2+t^2)^{1/2}},\quad t\in (0,\min\{x,1\}).
\end{align*}
According to \cite[Lemma 4.1]{ChHW} we obtain 
\[
\lim_{t\rightarrow 0^+}(J_{3,2,4}(t)+2\log t)=q+\mathcal X_{(0,1)}(x)\log x,
\]
where $q=2\int_1^\infty \left((r^2+1)^{-1/2}-r^{-1}\right)dr=2\log(2)-2\log(1+\sqrt{2})$. 

Combining the limits obtained for $J_{3,2,j}$ for $j=1,2,3, 4$, we conclude that 
\begin{align*}
\lim_{t\rightarrow 0^+}(J_{3,2}(t)+2\log t)&=\int_{B(x,1)}\int_1^\infty \frac{W_u^\lambda (x,y)}{u}dudy-\int_{B(x,1)}\int_1^\infty \frac{W_u(x-y)}{u}dudy\\
&\quad +\int_{B(x,1)}\int_0^1\frac{W_u^\lambda (x,y)-W_u(x-y)}{u}dudy+q+\mathcal X_{(0,1)}(x)\log x.
\end{align*}
Combining the above estimates, we get
\begin{align}\label{L2}
    \lim_{t\rightarrow 0^+}(I_3(t)+2f(x)\log t)&=\int_{B(x,1)}(f(y)-f(x))\int_0^\infty \frac{W_u^\lambda (x,y)}{u}dudy\nonumber\\
    &\quad +f(x)\left(\int_{B(x,1)}\int_1^\infty \frac{W_u^\lambda (x,y)}{u}dudy-\int_{B(x,1)}\int_1^\infty \frac{W_u(x-y)}{u}dudy\right.\nonumber\\
    &\quad +\left.\int_{B(x,1)}\int_0^1\frac{W_u^\lambda (x,y)-W_u(x-y)}{u}dudy+q+\mathcal X_{(0,1)}(x)\log x\right).
\end{align}

We now study $I_4$. We decompose it as follows
\begin{align*}
I_4(t)&=\int_{B(x,t)}f(y)\int_0^1\frac{e^{-\frac{t^2}{4u}}}{u}(W_u^\lambda (x,y)-W_u(x-y))dudy\\
&\quad +\int_{B(x,t)}f(y)\int_0^\infty \frac{e^{-\frac{t^2}{4u}}}{u}W_u (x-y)dudy\\
&\quad -\int_{B(x,t)}f(y)\int_1^\infty \frac{e^{-\frac{t^2}{4u}}}{u}W_u (x-y)dudy\\
&\quad +\int_{B(x,t)}f(y)\int_1^\infty \frac{e^{-\frac{t^2}{4u}}}{u}W_u^\lambda (x,y)dudy\\
&=\sum_{j=1}^4I_{4,j}(t),\quad t\in (0,1).
\end{align*}
Using \eqref{eq: heat Bessel FeynmanKac} we obtain
\[
|I_{4,3}(t)|+|I_{4,4}(t)|\leq C\int_{B(x,t)}|f(y)|\int_1^\infty \frac{e^{-c\frac{(x-y)^2+t^2}{u}}}{u^{3/2}}dudy\leq C\int_{B(x,t)}|f(y)|dy,\quad t\in (0,1).
\]
Since $f\in L^1_{\rm loc}(0,\infty)$ it follows that
\[
\lim_{t\rightarrow 0^+}(I_{4,3}(t)+I_{4,4}(t))=0.
\]
On the other hand, we have
\[
I_{4,2}(t)=\frac{1}{2\sqrt{\pi}}\int_{B(x,t)}f(y)\int_0^\infty \frac{e^{-\frac{(x-y)^2+t^2}{4u}}}{u^{3/2}}du=\int_{B(x,t)}\frac{f(y)}{((x-y)^2+t^2)^{1/2}}dy,\quad t\in (0,1).
\]
According to \cite[(4.50)]{ChHW} we get
\[
\lim_{t\rightarrow 0^+}I_{4,2}(t)=\widetilde{q}f(x),
\]
where $\widetilde{q}=2\int_0^1(1+t)^{-1/2}t^{-1/2}dt=4\log(1 + \sqrt{2})$.

For $I_{4,1}$, using \eqref{eq: heat Bessel FeynmanKac} and \eqref{eq: heat Bessel comparison t<xy} as above, we obtain
\begin{align*}
    |I_{4,1}(t)|&\leq \int_{B(x,t)}|f(y)|\left(\int_0^{\min\{xy,1\}}\frac{e^{-\frac{t^2}{4u}}}{u}|W_u^\lambda (x,y)-W_u(x-y)|du\right.\\
    &\quad +\left.\mathcal X_{(0,1/x)}(y)\int_{xy}^1\frac{e^{-\frac{t^2}{4u}}}{u}|W_u^\lambda (x,y)-W_u(x-y)|du\right)dy\\
    &\leq C\int_{B(x,t)}|f(y)|\left(\frac{1}{xy}\int_0^{\min\{xy,1\}}\frac{du}{u^{1/2}}+\mathcal X_{(0,1/x)}(y)\int_{xy}^\infty \frac{du}{u^{3/2}}\right)\\
    &\leq C\int_{B(x,t)}\frac{|f(y)|}{\sqrt{xy}}dy,\quad t\in (0,1).
\end{align*}
Then, since $f\in L^1_{\rm loc}(0,\infty)$, we deduce that
\[
\lim_{t\rightarrow 0^+}I_{4,1}(t)=0.
\]
Putting together the above limit results, we obtain
\begin{equation}\label{L3}
\lim_{t\rightarrow 0^+}I_4(t)=\widetilde{q}f(x).
\end{equation}

Combining \eqref{L0}--\eqref{L3} we conclude that
\begin{align*}
\lim_{t\rightarrow 0^+}(u_f^\lambda (x,t)+2f(x)\log t)
&=\int_{B(x,1)}(f(y)-f(x))\int_0^\infty \frac{W_u^\lambda (x,y)}{u}dudy\\
&\hspace{-3cm}+\int_{(0,\infty)\setminus B(x,1)}f(y)\int_0^\infty \frac{W_u^\lambda (x,y)}{u}dudy\\
&\hspace{-3cm}+f(x)\left(\int_{B(x,1)}\int_1^\infty \frac{W_u^\lambda (x,y)}{u}dudy-\int_{B(x,1)}\int_1^\infty \frac{W_u(x-y)}{u}dudy\right.\\
&\hspace{-3cm}\left.+\int_{B(x,1)}\int_0^1\frac{W_u^\lambda (x,y)-W_u(x-y)}{u}dudy+q+\widetilde{q}+\mathcal X_{(0,1)}(x)\log x\right).
\end{align*}
Here $q+\tilde{q}=2\log(2(1+\sqrt{2}))$.

Thus, we get
\begin{align*}
(\log B_\lambda)(f)(x)&=-\lim_{t\rightarrow 0^+}(u_f^\lambda (x,t)+2f(x)\log t)\\
&\quad +f(x)\Bigg(\int_{(0,\infty)\setminus B(x,1)}\int_0^1\frac{W_u(x-y)}{u}dudy-\int_{B(x,1)}\int_1^\infty \frac{W_u (x-y)}{u}dudy\\
&\quad -\gamma+q+\widetilde{q}+\mathcal X_{(0,1)}(x)\log x\Bigg)\\
&=-\lim_{t\rightarrow 0^+}(u_f^\lambda (x,t)+2f(x)\log t)\\
&\quad +f(x)\Big(\int_{(0,\infty)\setminus B(x,1)}\int_0^1\frac{W_u(x-y)}{u}dudy-\int_{B(x,1)}\int_1^\infty \frac{W_u (x-y)}{u}dudy\\
&\quad +2\log(2(1+\sqrt{2}))-\gamma+\mathcal X_{(0,1)}(x)\log(x)\Big).
\end{align*}

We are going to see that, for every $\varphi \in C_c^\infty(0,\infty)$,
\begin{equation}\label{distr}
\int_0^\infty f(x)(\log B_\lambda)\varphi (x)dx=\int_0^\infty \mathcal K(x)f(x)\varphi (x)dx-\lim_{t\rightarrow 0^+}\int_0^\infty (u_f^\lambda (x,t)+2f(x)\log t)\varphi (x)dx.
\end{equation}
Actually we will see that (\ref{distr}) holds when $\varphi$ is uniformly Dini continuous with compact support in $(0,\infty)$. Assume that $\varphi$ is uniformly Dini continuous with compact support in $(0,\infty)$. Let $t>0$. Using \eqref{eq: first estimate uflambda} we have that 
\begin{align*}
\int_0^\infty |\varphi (x)|\int_0^\infty |f(y)|&\int_0^\infty W_u^\lambda (x,y)\frac{e^{-\frac{t^2}{4u}}}{u}dudydx\\
&\leq C\int_{{\rm supp }\,\varphi}\left(\frac{1+x}{t}+\frac{1}{\min\{1,t\}}\right)\varphi (x)|dx\int_0^\infty \frac{|f(y)|}{1+y}dy<\infty.
\end{align*}
Then,
\[
\int_0^\infty u_f^\lambda (x,t)\varphi (x)dx=\int_0^\infty f(x)u_\varphi ^\lambda (x,t)dx,\quad t\in (0,\infty).
\]
We also have that
\[
(\log B_\lambda)(\varphi )(x)=-\lim_{t\rightarrow 0^+}(u_\varphi ^\lambda (x,t)+2f(x)\log t)+\varphi (x)\mathcal K(x),\quad x\in (0,\infty).
\]

Suppose $R>1$ and $\varphi (x)=0$, $x\in (R,\infty)$. Assume that $x>2R$. We get, using \eqref{eq: heat Bessel FeynmanKac}
\begin{align*}
|u_\varphi^\lambda (x,t)+2\varphi (x)\log t|=|u_\varphi^\lambda (x,t)|&\leq C\int_0^R|\varphi(y)|\int_0^\infty \frac{e^{-c\frac{(x-y)^2+t^2}{u}}}{u^{3/2}}dudy\\
&\leq C\int_0^R\frac{|\varphi (y)|}{((x-y)^2+t^2)^{1/2}}dy\leq C\int_0^R\frac{|\varphi (y)|}{|x-y|}dy\\
&\leq \frac{C}{x}\leq \frac{C}{1+x},\quad t\in (0,1).
\end{align*}

On the other hand, since $\varphi$ is uniformly Dini continuous in $(0,\infty)$ and there exists $0<a<R$ such that $\varphi (x)=0$, $x\in (0,a)$, we can see by proceeding as above that
\[
|u_\varphi ^\lambda (x,t)+2\varphi (x)\log t|\leq C,\quad x\in (0,2R)\mbox{ and }t\in (0,1).
\]
Then, 
\[
|u_\varphi ^\lambda (x,t)+2\varphi (x)\log t|\leq \frac{C}{1+x},\quad x\in (0,2R)\mbox{ and }t\in (0,1).
\]
By applying dominated convergence theorem we deduce that
\begin{align*}
    \lim_{t\rightarrow 0^+}\int_0^\infty (u_f(x,t)+2f(x)\log t)\varphi (x)dx&=\lim_{t\rightarrow 0^+}\int_0^\infty(u_\varphi ^\lambda (x,t)+\varphi (x)\log t)f(x)dx\\
    &=\int_0^\infty \lim_{t\rightarrow 0^+}(u_\varphi ^\lambda (x,t)+\varphi (x)\log t)f(x)dx\\
    &=-\int_0^\infty f(x)(\log B_\lambda)\varphi (x)dx+\int_0^\infty \mathcal K(x)f(x)\varphi (x)dx.
\end{align*}
Thus the proof is finished.

\section*{Statements and Declarations}

\subsection*{Authors' contributions} All authors contributed equally to the conceptualization, methodology, formal analysis, writing and review of this article.

\subsection*{Declaration of competing interest} The authors declare that they have no known competing financial interests or personal relationships that could have appeared to influence the work reported in this paper.

\subsection*{Funding} The authors are partially supported by the Grant PID2023-148028NB-I00 funded by MICIU/AEI/10.13039/501100011033 and by ERDF/EU.

\subsection*{Data availability} No data was used for the research described in the article.


\bibliographystyle{acm}
\def\cprime{$'$} \def\ocirc#1{\ifmmode\setbox0=\hbox{$#1$}\dimen0=\ht0 \advance\dimen0 by1pt\rlap{\hbox to\wd0{\hss\raise\dimen0 \hbox{\hskip.2em$\scriptscriptstyle\circ$}\hss}}#1\else {\accent"17 #1}\fi}

\end{document}